\documentclass[10 pt, onecolumn, oneside, a4paper, reqno]{amsart}
\usepackage{amssymb}
\usepackage{mathrsfs}
\usepackage{mathtools}
\usepackage{stmaryrd}
\usepackage{amsmath}
\usepackage{tikz-cd}
\usepackage{tikz}
\usepackage{pifont}
\usepackage{color}
\usepackage[all]{xy}
\usepackage{type1cm}
\usepackage{hyperref}
\usepackage{color}
\hypersetup{colorlinks=true,
     breaklinks=true,
     linkcolor=blue,
      pageanchor=true}

\newtheorem{thm}{Theorem}[section]

\newtheorem{cor}[thm]{Corollary}
\newtheorem{lem}[thm]{Lemma}
\newtheorem{prop}[thm]{Proposition}

\theoremstyle{definition}
\newtheorem{defn}[thm]{Definition}

\newtheorem{rem}[thm]{Remark}

\title{support $\tau$-tilting subcategories in exact categories}
\author{Jixing Pan, Yaohua Zhang and Bin Zhu}
\address{J. Pan: Department of  Mathematical Sciences, Tsinghua University, Beijing 100048, P. R. China.}
\email{pjx19@mails.tsinghua.edu.cn}
\address{Y. Zhang: Department of  Mathematical Sciences, Tsinghua University, Beijing 100048, P. R. China.}
\email{2160501008@cnu.edu.cn}
\address{B. Zhu: Department of  Mathematical Sciences, Tsinghua University, Beijing 100048, P. R. China.}
\email{zhu-b@mail.tsinghua.edu.cn}

\begin{document}

\begin{abstract}

Let $\mathcal{E}=(\mathcal{A},\mathcal{S})$ be an exact category with enough projectives $\mathcal{P}$. We introduce the notion of support $\tau$-tilting subcategories of $\mathcal{E}$. It is compatible with existing definitions of support $\tau$-tilting modules (subcategories) in various context. It is also a generalization of tilting subcategories of exact categories. We show that there is a bijection between support $\tau$-tilting subcategories and certain $\tau$-cotorsion pairs. Given a support $\tau$-tilting subcategory $\mathcal{T}$, we find a subcategory $\mathcal{E}_{\mathcal{T}}$ of $\mathcal{E}$ which is an exact category and $\mathcal{T}$ is a tilting subcategory of $\mathcal{E}_{\mathcal{T}}$. If $\mathcal{E}$ is Krull-Schmidt, we prove the cardinal $|\mathcal{T}|$ is equal to the number of isomorphism classes of indecomposable projectives $Q$ such that ${\rm Hom}_{\mathcal{E}}(Q,\mathcal{T})\neq 0$. We also show a functorial version of Brenner-Butler's theorem.

\end{abstract}

\maketitle

\section{Introduction}

Classical tilting theory, as a generalization of Morita equivalences, gives equivalences of certain subcategories of the categories of finitely generated modules over finite dimensional algebras. It started with the Coxeter functors defined by Bernstein, Gelfand and Ponomarev in \cite{BGP}, then was generalized by Auslander, Platzeck and Reiten in \cite{APR} and was axiomatically described by Brenner and Butler in \cite{Brenner and Butler}. Happel and Ringel \cite{HR} defined tilted algebras as endomorphism algebras of tilting modules over hereditary algebras. After that, tilting theory was generalized in many directions. Recently, in chapter 7 of \cite{H.Krause} the author defined tilting objects in exact categories. In \cite{Sauter} the author defined tilting subcategories in exact categories, unified several existing definitions and gave the concept of {\em ideq} tilting. In \cite{Zhu and Zhuang}, the authors defined tilting subcategories in extriangulated categories.

$\tau$-tilting theory of finite dimensional algebra was introduced by Adachi, Iyama and Reiten in \cite{AIR}. It can be seen as a generalization of classical tilting theory from the viewpoint of mutation. For more development, one can refer to \cite{Treffinger}. Later Iyama, Jørgensen and Yang defined support $\tau$-tilting subcategories in functor categories (see \cite{IJY}). And Angeleri Hügel, Marks and Vitória \cite{AMV} introduced the notion of silting modules in ${\rm Mod}-R$, which generalizes tilting modules over arbitrary rings and coincides with support $\tau$-tilting modules when restricting to finitely generated modules over a finite dimensional algebra.

Recently, Liu and Zhou \cite{Liu and Zhou} defined support $\tau$-tilting subcategories in Hom-finite abelian categories with enough projectives. Then Asadollahi, Sadeghi and Treffinger \cite{AST} modified the definitions of Liu and Zhou to drop the Hom-finite assumption. They showed a bijection between support $\tau$-tilting subcategories and $\tau$-cotorsion torsion triples (cf. \cite[Theorem 5.7]{AST}). It is a generalization of \cite[Theorem 2.29]{BBOS} and \cite[Theorem 4.6]{BZ}. They also characterize all support $\tau$-tilting subcategories of ${\rm Mod}-R$ by using finendo quasitilting modules defined in \cite{AMV}.

Motivated by the mentioned works, we introduce support $\tau$-tilting subcategories in exact categories with enough projectives and generalize some well-known results to exact categories. 

The paper is organized as follows. In Section 2, we provide some preliminaries about exact categories and functor categories and some notions.

In Section 3, let $\mathcal{E}$ be an exact category with enough projectives $\mathcal{P}$, we define support $\tau$-tilting subcategories of $\mathcal{E}$ as follows. 

\begin{defn}(Definition \ref{def})
	An additively closed subcategory $\mathcal{T}$ of $\mathcal{E}$ is a {\em support $\tau$-tilting} subcategory if it satisfies:
	\begin{enumerate}
		\item ${\rm Ext}^{1}_{\mathcal{E}}(\mathcal{T},\mathsf{Fac}\mathcal{T})=0$,
		\item For every $P\in \mathcal{P}$ there is an exact sequence
		$$P\stackrel{f} \rightarrow T^{0}\twoheadrightarrow T^{1}$$
		where $T^{0},T^{1}\in \mathcal{T}$ and $f$ is a left $\mathcal{T}$-approximation.
	\end{enumerate}
\end{defn}

It is a generalization of (1-) tilting subcategories in the sense of \cite{Sauter}. We do not require $\mathcal{T}$ to be contravariantly finite as in \cite{AST}. Then we extend some elementary results in \cite{AIR} and in abelian categories (\cite{AST} and \cite{Liu and Zhou}) to exact categories. Our first main theorem generalizes the bijection between support $\tau$-tilting subcategories and $\tau$-cotorsion torsion triples in abelian categories (cf. \cite[Theorem 5.7]{AST}) as follows. 
\begin{thm}(Theorem \ref{bijections})
	Assume $\mathcal{E}$ is weakly idempotent complete. Then there are mutually inverse bijections:
	\begin{align*}
		\{support~\tau\text{-tilting subcategories}\} & \leftrightarrow\{\tau\text{-cotorsion pair}~(\mathcal{C},\mathcal{D})~|~\mathcal{D}~\text{is a torsion class}\}\\
		\mathcal{T}& \mapsto (^{\bot_{1}}\mathsf{Fac}\mathcal{T},\mathsf{Fac}\mathcal{T})\\
		\mathcal{C} \cap \mathcal{D}& \mapsfrom (\mathcal{C},\mathcal{D}).
	\end{align*}
    Moreover the bijections restrict to bijections $$\{\text{tilting subcategories}\} \leftrightarrow \{\text{cotorsion pairs}~(\mathcal{C},\mathcal{D})~|~\mathcal{D}~\text{is a torsion class}\}.$$
\end{thm}

In Section 4, given a support $\tau$-tilting subcategory $\mathcal{T}$, we construct a subcategory $\mathcal{E}_{\mathcal{T}}\subseteq \mathcal{E}$ using the embedding $\mathcal{E}\rightarrow {\rm mod}_{\infty}-\mathcal{P}$. It is an exact category. Then we prove the first main result of this section. 
\begin{thm}(Theorem \ref{thm2})
	Assume $\mathcal{E}$ is skeletally small and weakly idempotent complete, then $\mathcal{T}$ is a tilting subcategory of $\mathcal{E}_{\mathcal{T}}$.
\end{thm}
If $\mathcal{E}={\rm mod}-\Lambda$ and $\mathcal{T}=\mathsf{add}T$ for a support $\tau$-tilting module $T$, then the above theorem is known in \cite{AIR} and \cite{Treffinger2}. If moreover $\mathcal{E}$ is Krull-Schmidt, we prove another main result which is well-known in ${\rm mod}-\Lambda$ (see \cite{AIR}) as follows. 
\begin{thm}(Theorem \ref{thm:number component})
	$|\mathcal{T}|$ equals to the number of isomorphism classes of indecomposable projectives $Q$ such that ${\rm Hom}_{\mathcal{E}}(Q,\mathcal{T})\neq 0$.
\end{thm}

In Section 5, we give a Brenner-Butler-type theorem in functor categories (see Theorem \ref{thm4}). It is a functorial version of \cite[Proposition 3.5]{G.Jasso} and partially generalizes \cite[Theorem 6.13]{Sauter}.

In Section 6, we provide a simple example to illustrate some of our results.

\section{Preliminaries}

In this paper, all categories are additive and subcategories are always assumed to be full and closed under isomorphisms. Let $\mathcal{A}$ be an additive category. A subcategory $\mathcal{C}$ of $\mathcal{A}$ is called {\em additively closed} if it is closed under taking direct summands of finite direct sums. The composition of $f:X\rightarrow Y$ and $g:Y\rightarrow Z$ is $gf$.
A morphism $f:X\rightarrow Y$ is {\em right minimal} if any morphism $h:X\rightarrow X$ satisfies $fh=f$ is an isomorphism. For a subcategory $\mathcal{X}$ of $\mathcal{A}$, $f$ is a {\em right $\mathcal{X}$-approximation} if $X\in \mathcal{X}$ and ${\rm Hom}_{\mathcal{A}}(X',f)$ is surjective for any $X'\in \mathcal{X}$.  $f$ is called a {\em minimal right $\mathcal{X}$-approximation} if $f$ is both right minimal and a right approximation. Dually we define (minimal) left approximations. 

An additive category $\mathcal{A}$ is called {\em weakly idempotent complete} if every section has a cokernel (equivalently, every retraction has a kernel, cf. \cite[Lemma 7.1]{T.Buhler}). It is called {\em idempotent complete} (or has split idempotents) if for every idempotent $e\in {\rm End}_{\mathcal{A}}(X)$, there is $u:X\rightarrow Y$ and $v:Y\rightarrow X$ such that $vu=e,uv={\rm id}_{Y}$. It is called {\em Krull-Schmidt} if every object decomposes into a finite direct sum of objects having local endomorphism rings.

For a unital ring $A$, we denote by ${\rm Mod}-A$ the category of right $A$-modules, and ${\rm mod}-A$ the subcategory of finitely generated modules.

\subsection{Exact categories}

Let $\mathcal{E}=(\mathcal{A},\mathcal{S})$ be an exact category, where $\mathcal{A}$ is an additive category and $\mathcal{S}$ is a class of kernel-cokernel pairs which satisfies axioms of \cite[Definition 2.1]{T.Buhler}. We call an element $X\stackrel{f}\rightarrow Y\stackrel{g}\rightarrow Z$ in $\mathcal{S}$ a conflation, $f$ an inflation and $g$ a deflation. We depict inflations (resp. deflations) by $\rightarrowtail$ (resp. $\twoheadrightarrow$). A morphism $f$ in $\mathcal{E}$ is called {\em admissible} if $f=i\circ d$ for some inflation $i$ and deflation $d$. A sequence of admissible morphisms $X\stackrel{f} \rightarrow Y\stackrel{g} \rightarrow Z$ in $\mathcal{E}$ is {\em exact} at $Y$ if Im$f$=Ker$g$. A sequence of composable morphisms is exact if every morphism is admissible and the sequence is exact at every intermediate object.

For a class of objects $\mathcal{C}$, denote by $\mathsf{add}\mathcal{C}$ the smallest additively closed subcategory containing $\mathcal{C}$. For an additive subcategory $\mathcal{T}$, define
$$\mathsf{Fac}\mathcal{T}=\{X\in \mathcal{E}~|~\exists~T\twoheadrightarrow X,~T\in \mathcal{T}\}~\text{and}~\mathsf{Sub}\mathcal{T}=\{X\in \mathcal{E}~|~\exists~X\rightarrowtail T,~T\in \mathcal{T}\}.$$
We say $\mathcal{T}$ is {\em factor closed} (resp. {\em subobject closed}) if $\mathsf{Fac}\mathcal{T}\subseteq \mathcal{T}$ (resp. $\mathsf{Sub}\mathcal{T}\subseteq \mathcal{T}$).

An object $P\in \mathcal{E}$ is called {\em projective} if the functor Hom$(P,-)$ : $\mathcal{E}\longrightarrow $Ab takes conflations to conflations. We say $\mathcal{E}$ has enough projectives if for every object $X$ there is a deflation $P\twoheadrightarrow X$ with $P$ projective. Dually, we have {\em injective} objects.

As in abelian case, there are abelian groups ${\rm Ext}^{n}_{\mathcal{E}}(X,Y)$ for $n\geq 0$ and $X,Y\in \mathcal{E}$. They can be defined by using projective (or injective) resolutions, or as abelian groups of $n$-extensions in the sense of Yoneda or Hom spaces in derived categories. We refer to \cite{H.Krause} for details. For a subcategory $\mathcal{X}$, we define
$$\mathcal{X}^{\bot_{n}}=\{Y\in \mathcal{E}~|~{\rm Ext}^{n}_{\mathcal{E}}(X,Y)=0~\forall X\in \mathcal{X}\}.$$
Clearly we have $\mathcal{X}^{\bot_{0}}=\{Y\in \mathcal{E}~|~{\rm Hom}_{\mathcal{E}}(X,Y)=0~\forall X\in \mathcal{X}\}$. Dually we can define $^{\bot_{n}}\mathcal{X}$. For an object $X\in \mathcal{E}$, we say ${\rm pd}_{\mathcal{E}}X\leq n$ if ${\rm Ext}_{\mathcal{E}}^{n+1}(X,-)=0$.

A subcategory $\mathcal{C}$ of $\mathcal{E}$ is called {\em extension closed} if for any conflation $X\rightarrowtail Y\twoheadrightarrow Z$ with $X,Z\in \mathcal{C}$, then $Y\in \mathcal{C}$. It is called a {\em thick} subcategory if it is closed under extensions, kernels of deflations, cokernels of inflations and summands. We denote by ${\rm Thick}(\mathcal{C})$ the smallest thick subcategory containing $\mathcal{C}$.

A {\em full exact subcategory} $\mathcal{C}$ of $\mathcal{E}$ is a full extension closed additive subcategory. In this case, $(\mathcal{C},\mathcal{S}|_{\mathcal{C}})$ is an exact category. Note that the subcategory ${\rm P}(\mathcal{C})$ consisting of Ext-projective objects in $\mathcal{C}$ (i.e. $X\in \mathcal{C}$ such that ${\rm Ext}_{\mathcal{E}}^{1}(X,\mathcal{C})=0$) is precisely the subcategory of projective objects in exact category $\mathcal{C}$.

Exact category $\mathcal{E}=(\mathcal{A},\mathcal{S})$ is weakly idempotent complete (resp. idempotent complete, Krull-Schmidt) if the underlying additive category $\mathcal{A}$ is weakly idempotent complete (resp. idempotent complete, Krull-Schmidt). When $\mathcal{E}$ is weakly idempotent complete, there is a Heller's cancellation axiom as follows.

\begin{lem}\cite[Proposition 7.6]{T.Buhler}\label{deflation}
	Let $\mathcal{E}=(\mathcal{A},\mathcal{S})$ be an exact category. The following are equivalent:
	\begin{enumerate}
		\item The additive category $\mathcal{A}$ is weakly idempotent complete.
		\item Consider two morphisms $g:B\rightarrow C$ and $f:A\rightarrow B$. If $gf:A\twoheadrightarrow C$ is a deflation then $g$ is a deflation.
	\end{enumerate}
\end{lem}

If $\mathcal{E}$ is idempotent complete, then there is a useful lemma.

\begin{lem}\cite[Exercise 8.18]{T.Buhler}\label{admissible}
	Let $\mathcal{E}=(\mathcal{A},\mathcal{S})$ be an exact category. The following are equivalent:
	\begin{enumerate}
		\item The category $\mathcal{A}$ is idempotent complete.
		\item If the direct sum of two morphisms $a:A'\rightarrow A$ and $b:B'\rightarrow B$ is an admissible morphism, then so are $a$ and $b$.
	\end{enumerate}
\end{lem}

Exact categories have many diagram properties such as Five Lemma \cite[Corollary 3.2]{T.Buhler}. For more, we refer to \cite{T.Buhler}.

\subsection{Functor categories}

Let $\mathcal{C}$ be an skeletally small additive category, define ${\rm Mod}-\mathcal{C}$ to be the functor category consisting of all ({\em right}) $\mathcal{C}$-{\em modules}, i.e., contravariant additive functors $F:\mathcal{C}^{\rm op}\longrightarrow Ab$. Then ${\rm Mod}-\mathcal{C}$ is an abelian category with enough projectives and injectives, and projectives are precisely direct summands of direct sums (possibly infinite) of representable functors. We denote by ${\rm mod}_{n}-\mathcal{C}$ ($n\in \mathbb{N}\cup \{\infty\}$) the subcategory of $\mathcal{C}$-modules $F$ which admits an exact sequence
$$\mathcal{C}(-,X_{n})\rightarrow \cdots \rightarrow \mathcal{C}(-,X_{1})\rightarrow \mathcal{C}(-,X_{0})\rightarrow F\rightarrow 0, ~X_i\in \mathcal{C}.$$
Then ${\rm mod}_{\infty}-\mathcal{C}$ is a thick subcategory of ${\rm Mod}-\mathcal{C}$ and has enough projectives which is the subcategory of all finitely generated projectives in ${\rm Mod}-\mathcal{C}$ (cf. \cite[Proposition 2.6]{H.Enomoto}). If furthermore, $\mathcal{C}$ is idempotent complete, then finitely generated projectives are precisely reprensentable functors (cf. \cite[Proposition 2.2]{Auslander}). Let $A$ be a ring, we have ${\rm Mod}-A$ and ${\rm mod}_{n}-A,~n\in \mathbb{N}\cup \{\infty\}$ similarily. ${\rm Mod}-A$ is just the category of right $A$-modules and ${\rm mod}_{0}-A$ (resp. ${\rm mod}_{1}-A$) is the subcategory of finitely generated (resp. presented) modules. Note that if $A$ is right noetherian, then ${\rm mod}_{n}-A={\rm mod}_{0}-A,~\forall n\in \mathbb{N}\cup \{\infty\}$.

\begin{lem}\label{infinite resolution}
	Let $\mathcal{C}$ be a skeletally small additive category. If $F\in {\rm Mod}-\mathcal{C}$ admits an exact sequence
	$$\cdots \rightarrow G_{2} \stackrel{\alpha_2}\rightarrow G_{1} \stackrel{\alpha_1}\rightarrow G_{0} \stackrel{\alpha}\rightarrow F\rightarrow 0,~ G_{i}\in {\rm mod}_{\infty}-\mathcal{C},$$
	then $F\in {\rm mod}_{\infty}-\mathcal{C}$.
\end{lem}
\begin{proof}
	Because $G_{0}\in {\rm mod}_{\infty}-\mathcal{C}$, there is an exact sequence $$0 \rightarrow H_{0}\stackrel{\beta} \rightarrow C_{0} \stackrel{\gamma}\rightarrow G_{0} \rightarrow 0$$ 
	where $C_{0}$ is a representable functor and $H_{0}\in {\rm mod}_{\infty}-\mathcal{C}$. Consider the following pullback diagram
	\[\begin{tikzcd}
		&  & H_{0} \arrow[r,equal] \arrow[d] & H_{0} \arrow[d, "\beta"] &  &  \\
		\cdots \arrow[r] & G_{2} \arrow[r] \arrow[d,equal] & G_{1}' \arrow[r] \arrow[d] \arrow[dr, phantom, "{\rm PB}"] & C_{0} \arrow[r] \arrow[d, "\gamma"] & F \arrow[r] \arrow[d,equal] & 0 \\
		\cdots \arrow[r] & G_{2} \arrow[r, "\alpha_2"] & G_{1} \arrow[r, "\alpha_1"] & G_{0} \arrow[r,"\alpha"] & F \arrow[r] & 0
	\end{tikzcd}\]
    Since ${\rm mod}_{\infty}-\mathcal{C}$ is extension closed and $H_0, G_1\in {\rm mod}_{\infty}-\mathcal{C}$, then $G_{1}'\in {\rm mod}_{\infty}-\mathcal{C}$. Repeat the construction above, then one can obtain an exact sequence
    $$\cdots C_2\rightarrow C_1\rightarrow C_0\rightarrow F\rightarrow 0$$
    with $C_i$ representable. This finishes the proof.
\end{proof}

Assume that skeletally small exact category $\mathcal{E}=(\mathcal{A},\mathcal{S})$ has enough projectives $\mathcal{P}$. We can define a functor
\[\mathbb{P}:~\mathcal{E}\longrightarrow {\rm Mod}-\mathcal{P}, ~~X\mapsto \mathcal{E}(-,X)|_{\mathcal{P}}\]
whose image lies in ${\rm mod}_{\infty}-\mathcal{P}$ since $\mathcal{E}$ has enough projectives. We have the following result.

\begin{lem}\cite[Proposition 2.1, 2.8]{H.Enomoto} \label{Enomoto}
	The functor $\mathbb{P}:~\mathcal{E}\longrightarrow {\rm Mod}-\mathcal{P}$ is fully faithful, exact, preserves all extension groups and induces exact equivalence $\mathbb{P}:~\mathcal{E}\stackrel{\simeq} \longrightarrow {\rm Im}\mathbb{P}$. Moreover if $\mathcal{E}$ is idempotent complete, then ${\rm Im}\mathbb{P}$ is a resolving subcategory of ${\rm mod}_{\infty}-\mathcal{P}$.
\end{lem}

\subsection{Tilting subcategories in exact categories}

\begin{defn}\cite[Definition 4.1]{Sauter}
	Let $\mathcal{E}=(\mathcal{A},\mathcal{S})$ be an exact category. A subcategory $\mathcal{T}$ of $\mathcal{E}$ is {\em n-tilting} if it satisfies: 
	\begin{enumerate}
		\item $\mathcal{T}^{\bot}:=\bigcap\limits_{i=1}^{\infty}\mathcal{T}^{\bot_{i}}$ has enough projectives $\mathcal{T}$, 
		\item ${\rm Cores}_{n}(\mathcal{T}^{\bot})=\mathcal{E}$ (i.e. $\forall X\in \mathcal{E},\exists {\rm~exact~sequence}~X\rightarrowtail Y_{0}\rightarrow \cdots \rightarrow Y_{n-1}\twoheadrightarrow Y_{n},~Y_{i}\in \mathcal{T}^{\bot}$).
	\end{enumerate}
\end{defn}

If $\mathcal{E}$ has enough projectives, then the definition coincides with the usual one. 

\begin{lem}\cite[Theorem 5.3]{Sauter}
	Let $\mathcal{E}=(\mathcal{A},\mathcal{S})$ be an exact category with enough projectives $\mathcal{P}$. An additive subcategory $\mathcal{T}$ is $n$-tilting if and only if it satisfies:
	\begin{enumerate}
		\item $\mathcal{T}$ is closed under summands and self-orthogonal (i.e. $\mathcal{T} \subseteq \mathcal{T}^{\bot}$),
		\item ${\rm pd}_{\mathcal{E}}\mathcal{T}\leq n$,
		\item $\mathcal{P} \subseteq {\rm Cores}_{n}(\mathcal{T})$. 
	\end{enumerate}
\end{lem}

The default meaning of "tilting" is "1-tilting". Clearly, $\mathcal{P}$ is a 0-tilting subcategory. For convenience, if $n=1$ and $\mathcal{T}$ satisfies (1) and (2), then we call it a {\em partial tilting} subcategory.

\subsection{Torsion and ($\tau$-) cotorsion pairs in exact categories}

\begin{defn}
	Let $\mathcal{E}=(\mathcal{A},\mathcal{S})$ be an exact category.  A pair of subcategories $(\mathcal{C},\mathcal{D})$ of $\mathcal{E}$ is called a {\em torsion pair} if it satisfies:
	\begin{enumerate}
		\item ${\rm Hom}_{\mathcal{E}}(\mathcal{C},\mathcal{D})=0$
		\item For every $X\in \mathcal{E}$, there exists a conflation 
		\[C\rightarrowtail X\twoheadrightarrow D\]
		with $C\in \mathcal{C},~D\in \mathcal{D}$.
	\end{enumerate}
\end{defn}

For a torsion pair $(\mathcal{C},\mathcal{D})$, we have $\mathcal{C}={^{\bot_{0}}\mathcal{D}},~\mathcal{D}=\mathcal{C}^{\bot_{0}}$. Hence $\mathcal{C}$ (resp. $\mathcal{D}$) is factor (resp. subobject) closed and extension closed. We call a subcategory of $\mathcal{E}$ a {\em torsion} (resp. {\em torsion free}) {\em class} if it is factor (resp. subobject) closed and extension closed.

\begin{defn}\cite[Lemma 1.6]{BZ}
	Let $\mathcal{E}=(\mathcal{A},\mathcal{S})$ be an exact category. A pair of subcategories $(\mathcal{C},\mathcal{D})$ of $\mathcal{E}$ is called a ({\em complete}) {\em cotorsion pair} if it satisfies:
	\begin{enumerate}
		\item $\mathcal{C},\mathcal{D}$ are closed under direct summands,
		\item ${\rm Ext}_{\mathcal{E}}^{1}(\mathcal{C},\mathcal{D})=0$,
		\item For each object $X\in \mathcal{E}$, there exist conflations
		$$D\rightarrowtail C\twoheadrightarrow X,~C\in \mathcal{C},D\in \mathcal{D}$$
		and 
		$$X\rightarrowtail D^{'}\twoheadrightarrow C^{'},~C'\in \mathcal{C},D'\in \mathcal{D}.$$
	\end{enumerate}
\end{defn}

\begin{defn}\cite[Definition 4.1]{AST}
	Let $\mathcal{E}=(\mathcal{A},\mathcal{S})$ be an exact category with enough projectives $\mathcal{P}$. A pair of subcategories $(\mathcal{C},\mathcal{D})$ of $\mathcal{E}$ is called a {\em $\tau$-cotorsion pair} if it satisfies:
	\begin{enumerate}
		\item $\mathcal{C} ={^{\bot_1}\mathcal{D}}$,
		\item For every $P\in \mathcal{P}$, there is an exact sequence 
		$$P\stackrel{f} \rightarrow D\twoheadrightarrow C$$ 
		where $D\in \mathcal{C} \cap \mathcal{D}, C\in \mathcal{C}$ and $f$ is a left $\mathcal{D}$-approximation.
	\end{enumerate}
\end{defn}

Cotorsion pairs are $\tau$-cotorsion pairs since every projective object is in $\mathcal{C}$ and $\mathcal{C}$ is extension closed.

\begin{defn}
	Let $\mathcal{E}=(\mathcal{A},\mathcal{S})$ be an exact category (with enough projectives). A triple $(\mathcal{C},\mathcal{D},\mathcal{F})$ is called a ($\tau$-) {\em cotorsion torsion triple} if $(\mathcal{C},\mathcal{D})$ is a ($\tau$-) cotorsion pair and $(\mathcal{D},\mathcal{F})$ is a torsion pair.
\end{defn}

\section{Support $\tau$-tilting subcategories}

In this section. we always assume that $\mathcal{E}=(\mathcal{A},\mathcal{S})$ is an exact category with enough projectives $\mathcal{P}$.

\begin{defn}\label{def}
	An additively closed subcategory $\mathcal{T}$ of $\mathcal{E}$ is a {\em support $\tau$-tilting} subcategory if it satisfies:
	\begin{enumerate}
		\item ${\rm Ext}^{1}_{\mathcal{E}}(\mathcal{T},\mathsf{Fac}\mathcal{T})=0$,
		\item For every $P\in \mathcal{P}$ there is an exact sequence
		$$P\stackrel{f} \rightarrow T^{0}\twoheadrightarrow T^{1}$$
		where $T^{0},T^{1}\in \mathcal{T}$ and $f$ is a left $\mathcal{T}$-approximation.
	\end{enumerate}
\end{defn}

It is {\em $\tau$-rigid} if (1) is satisfied. A support $\tau$-tilting subcategory is {\em $\tau$-tilting } if for each $0\neq P\in \mathcal{P}$ there is a nonzero $f$ satisfying (2). We call an object $T$ a support $\tau$-tilting (resp. $\tau$-rigid, $\tau$-tilting, tilting) object if $\mathsf{add}T$ is a support $\tau$-tilting (resp. $\tau$-rigid, $\tau$-tilting, tilting) subcategory.

\begin{rem}
(1) When $\mathcal{E}$ is an abelian category, support $\tau$-tilting subcategories are just weak support $\tau$-tilting subcategories in the sense of \cite{AST}. If moreover $\mathcal{E}$ is $k$-linear ($k$ is a field) Hom-finite, then our definition coincides with that in \cite{Liu and Zhou}. 

(2) If $\mathcal{E}={\rm mod}-\Lambda$ for an artin algebra $\Lambda$, our definition coincides with that in \cite{AIR} by \cite[Proposition 5.8]{AS81} and \cite[Proposition 2.14]{G.Jasso}.

(3) Tilting subcategories are support $\tau$-tilting subcategories by \cite[Lemma 4.4]{Sauter}.
\end{rem}

\begin{lem}\label{criterion}
	Let  $\mathcal{T}$  be a support $\tau$-tilting subcategory of $\mathcal{E}$. Then $\mathcal{T}$ is tilting if and only if for each $P\in \mathcal{P}$ there is an inflation $f$ satisfying Definition~\ref{def} (2).
\end{lem}
\begin{proof}
	The necessity is clear. For the sufficiency, we need to show ${\rm pd}_{\mathcal{E}}\mathcal{T}\leq 1$. It suffices to prove ${\rm Ext}_{\mathcal{E}}^{2}(T,X)=0$ for $T\in\mathcal{T}, X\in \mathcal{E}$. Consider an exact sequence $$X\stackrel{a}\rightarrowtail Z_{2}\stackrel{b}\rightarrow Z_{1}\stackrel{c}\twoheadrightarrow T.$$ Because $\mathcal{E}$ has enough projectives, there exists a deflation $P\twoheadrightarrow Z_{2}$ with $P\in \mathcal{P}$. By assumption, $P$ admits a conflation
	$$P\stackrel{f}\rightarrowtail T^{0}\twoheadrightarrow T^{1},~T^0, T^1\in\mathcal{T}.$$
	Then there is a pushout diagram
	\[\begin{tikzcd}
		P \arrow[r,tail, "f"] \arrow[d,two heads] \arrow[dr, phantom, "{\rm PO}"]
		& T^{0} \arrow[r,two heads] \arrow[d,two heads, "h"] &T^{1} \arrow[d,equal]\\
		Z_{2} \arrow[r,tail, "g"] & M \arrow[r,two heads] &T^{1}
	\end{tikzcd}\]
    in which $h$ is a deflation by Five Lemma (\cite[Corollary 3.2]{T.Buhler}) and $g$ is an inflation.
    Consider the pushout of $b$ along $g$, there is a commutative diagram
    \[\begin{tikzcd}
    	X \arrow[r,tail,"a"] \arrow[d,equal] & Z_{2} \arrow[r, "b"] \arrow[d,tail, "g"] \arrow[dr, phantom, "{\rm PO}"] & Z_{1} \arrow[d,tail] \arrow[r,two heads,"c"] & T \arrow[d,equal]\\
    	X \arrow[r,tail] & M \arrow[r] & N \arrow[r,two heads] & T.
    \end{tikzcd}\]
    Since $h$ is a deflation, $M\in \mathsf{Fac}\mathcal{T}$. This implies the second row in the last diagram is an exact sequence whose class is zero in the group ${\rm Ext}_{\mathcal{E}}^{2}(T,X)$ and so is the first row. Thus we finish the proof.
\end{proof}

The following result is well-known in \cite{AIR}. A weaker result is also proved in abelian categories (see \cite[Lemma 5.1]{AST} and \cite[Corollary 3.3]{Liu and Zhou}). 

\begin{lem}\label{enough proj T}
	Assume $\mathcal{E}$ is weakly idempotent  complete. Let $\mathcal{T}$ be a support $\tau$-tilting subcategory of $\mathcal{E}$. Then $\mathsf{Fac}\mathcal{T}$ is a full exact subcategory of $\mathcal{E}$ and has enough projectives $\mathcal{T}$.
\end{lem}
\begin{proof}
	$\mathsf{Fac}\mathcal{T}$ is a full exact subcategory by Horseshoe lemma.
	Clearly, objects in $\mathcal{T}$ are projective in exact category $\mathsf{Fac}\mathcal{T}$. For every $X\in \mathsf{Fac}\mathcal{T}$ we have a conflation 
	$$Y\stackrel{u}\rightarrowtail T\stackrel{v}\twoheadrightarrow X~\text{for some}~T\in \mathcal{T}.$$
	Because $\mathcal{E}$ has enough projectives, there is a deflation $g:P\twoheadrightarrow Y$ with $P\in \mathcal{P}$, and an exact sequence 
	$$P\stackrel{f}\rightarrow T^0\twoheadrightarrow T^1,~T^0, T^1\in \mathcal{T}$$ with $f$ a left $\mathcal{T}$-approximation.
	Decompose $f$ to a composition of a deflation $j$ and an inflation $i$, we obtain a commutative diagram
	\[\begin{tikzcd}
		P \arrow[r, "f"]\arrow[dd, bend right, "g" swap,two heads]
		\arrow[d, two heads, "j"] & T^0\arrow[r]\arrow[d,equal] & T^1 \arrow[d,equal]\\
		K \arrow[r, tail, "i"] \arrow[d, dashrightarrow, "a"] & T^{0} \arrow[r,two heads] \arrow[d, dashrightarrow, "b"] &T^{1} \arrow[d, dashrightarrow, "c"]\\
		Y \arrow[r,tail, "u"] & T \arrow[r,two heads, "v"] & X
	\end{tikzcd}\]
    where $b,c$ exist since $f$ is an approximation. This implies the existence of $a$. Since $ug=bf=bij=uaj$, then $g=aj$. Thus $a$ is a deflation by Lemma \ref{deflation}. Consider the pushout diagram 
    \[\begin{tikzcd}
    	K \arrow[r,tail, "i"] \arrow[d,two heads,"a" swap] \arrow[dr, phantom, "{\rm PO}"]
    	& T^{0} \arrow[r,two heads] \arrow[d,two heads] &T^{1} \arrow[d,equal]\\
    	Y \arrow[r,tail, "e"] & D \arrow[r,two heads, "l"] &T^{1},
    \end{tikzcd}\]
    we obtain $D\in \mathsf{Fac}\mathcal{T}$. Then consider the following diagram
    \[\begin{tikzcd}
    	Y \arrow[r,tail, "u"] \arrow[d,tail, "e" swap] \arrow[dr, phantom, "{\rm PO}"]
    	& T \arrow[r,two heads, "v"] \arrow[d,tail] &X \arrow[d,equal]\\
    	D \arrow[r,tail] \arrow[d,two heads, "l" swap] & Z \arrow[r,two heads] \arrow[d,two heads] &X\\
    	T^{1} \arrow[r,equal] & T^{1}. &
    \end{tikzcd}\]
    The mid-column is split since ${\rm Ext}_{\mathcal{E}}^{1}(T^1,T)=0$. Thus $Z\in\mathcal{T}$ is projective and the mid-row is a desired conflation.
\end{proof}

Let $\mathcal{T}_{1}$ and $ \mathcal{T}_{2}$ be two support $\tau$-tilting subcategories of $\mathcal{E}$. Define $\mathcal{T}_{1}\leq \mathcal{T}_{2}$ if  $\mathsf{Fac}\mathcal{T}_{1} \subseteq \mathsf{Fac}\mathcal{T}_{2}$.

\begin{cor}\label{poset}
	Assume $\mathcal{E}$ is weakly idempotent  complete. Then $\leq$ is a partial order on the collection of all support $\tau$-tilting subcategories.
\end{cor}
\begin{proof}
	Clear.
\end{proof}

We give a condition for a ($\tau$-rigid) subcategory $\mathcal{T}$:

\begin{center}
   {\bf (A)}: $\forall P\in \mathcal{P}$ there is a left $\mathcal{T}$-approximation $P\stackrel{f}\rightarrow T$ of $P$ with $f$ admissible. 
\end{center}

Support $\tau$-tilting subcategories are $\tau$-rigid subcategories satisfying (A). A $\tau$-rigid module $T$ in ${\rm mod}-\Lambda$ ($\Lambda$ is an artin algebra) satisfies (A) automatically (consider $\mathsf{add}T$-approximation instead). The following Proposition generalizes \cite[Theorem 1.1]{Liu and Zhou}.

\begin{prop}\label{inclusion}
	Let $\mathcal{T} \subseteq \mathcal{E}$ be a $\tau$-rigid subcategory satisfying (A). Then there exists a support $\tau$-tilting subcategory $\mathcal{T}^{'}$ such that $\mathcal{T} \subseteq \mathcal{T}^{'}$. 
\end{prop}
\begin{proof}
	Since $\mathcal{T}$ is $\tau$-rigid and by Horseshoe Lemma, $\mathsf{Fac}\mathcal{T}$ is a full exact subcategory of $\mathcal{E}$. Let $\mathcal{T}^{'}={\rm P}(\mathsf{Fac}\mathcal{T})$. Clearly we have $\mathcal{T} \subseteq \mathcal{T}^{'}$ and ${\rm Ext}^{1}_{\mathcal{E}}(\mathcal{T}^{'},\mathsf{Fac}\mathcal{T}^{'})=0$. For each $P\in \mathcal{P}$, there is an exact sequence (by (A))
	$$P\stackrel{f} \rightarrow T \twoheadrightarrow T^{'},~T\in \mathcal{T}$$ 
	with $f$ a left $\mathcal{T}$-approximation. By decomposing $f$ as $P \stackrel{d}\twoheadrightarrow K \stackrel{i}\rightarrowtail T$ we obtain a conflation 
	$$K \stackrel{i}\rightarrowtail T \twoheadrightarrow T^{'}.$$ 
	For any $S\in \mathsf{Fac}\mathcal{T}$, apply ${\rm Hom}_{\mathcal{E}}(-,S)$ to the last conflation then we obtain an exact sequence 
	\[0\rightarrow {\rm Hom}_{\mathcal{E}}(T',S) \rightarrow {\rm Hom}_{\mathcal{E}}(T,S) \stackrel{i^{*}}\rightarrow {\rm Hom}_{\mathcal{E}}(K,S) \rightarrow {\rm Ext}_{\mathcal{E}}^{1}(T',S) \rightarrow {\rm Ext}_{\mathcal{E}}^{1}(T,S)\]
	Since $f$ is also a left $\mathsf{Fac}\mathcal{T}$-approximation, so is $i$. Thus $i^{*}$ is surjective. Because ${\rm Ext}_{\mathcal{E}}^{1}(T,S)=0$, ${\rm Ext}_{\mathcal{E}}^{1}(T',S)=0$ for all $S\in \mathsf{Fac}\mathcal{T}$. Thus $T^{'}\in {\rm P}(\mathsf{Fac}\mathcal{T})=\mathcal{T}^{'}$. Hence $\mathcal{T}^{'}$ is a support $\tau$-tilting subcategory.
\end{proof}

\begin{cor}\label{T=Ext-proj}
	A $\tau$-rigid subcategory $\mathcal{T}$ satisfying (A) is a support $\tau$-tilting subcategory if and only if $\mathcal{T}={\rm P}(\mathsf{Fac}\mathcal{T})$.
\end{cor}
\begin{proof}
	Clear.
\end{proof}

The above result provides us a description of support $\tau$-tilting modules in ${\rm mod}-\Lambda$. For a module $T$, if there is a full exact subcategory $\mathcal{E}$ of ${\rm mod}-\Lambda$  having enough projectives $\mathsf{add}T$, then $\mathcal{E}\subseteq \mathsf{Fac}T$. $\mathsf{Fac}T$ is such a subcategory if and only if $T$ is a support $\tau$-tilting module.

Now we prove the main result of this section which generalizes \cite[Theorem 1.2]{Liu and Zhou} and \cite[Theorem 5.7]{AST}. The latter generalizes \cite[Theorem 2.29]{BBOS} and \cite[Theorem 4.6]{BZ}.

\begin{thm}\label{bijections}
	Asuume $\mathcal{E}$ is weakly idempotent complete. Then there are mutually inverse bijections: 
	\begin{align*}
		\{support~\tau\text{-tilting subcategories}\} & \leftrightarrow\{\tau\text{-cotorsion pair}~(\mathcal{C},\mathcal{D})~|~\mathcal{D}~\text{is a torsion class}\}\\
		\mathcal{T}& \mapsto (^{\bot_{1}}\mathsf{Fac}\mathcal{T},\mathsf{Fac}\mathcal{T})\\
		\mathcal{C} \cap \mathcal{D}& \mapsfrom (\mathcal{C},\mathcal{D}).
	\end{align*}
    Moreover the bijections restrict to bijections $$\{\text{tilting subcategories}\} \leftrightarrow \{\text{cotorsion pairs}~(\mathcal{C},\mathcal{D})~|~\mathcal{D}~\text{is a torsion class}\}.$$
\end{thm}
\begin{proof}
	If $\mathcal{T}$ is a support $\tau$-tilting subcategory, for every $P\in \mathcal{P}$ there is an exact sequence 
	$$P\stackrel{f} \rightarrow T^{0}\twoheadrightarrow T^{1},~T^{0},T^{1}\in \mathcal{T}$$ 
	with $f$ a left $\mathcal{T}$-approximation. Indeed $f$ is also a left $\mathsf{Fac}\mathcal{T}$-approximation, so $(^{\bot_{1}}\mathsf{Fac}\mathcal{T},\mathsf{Fac}\mathcal{T})$ is a $\tau$-cotorsion pair with $\mathsf{Fac}\mathcal{T}$ a torsion class (see Lemma \ref{enough proj T}). 
	
	If $(\mathcal{C},\mathcal{D})$ is a $\tau$-cotorsion pair with $\mathcal{D}$ a torsion class, it is obvious that $\mathcal{C} \cap \mathcal{D}$ is additively closed and ${\rm Ext}_{\mathcal{E}}^{1}(\mathcal{C} \cap \mathcal{D},\mathsf{Fac}(\mathcal{C} \cap \mathcal{D}))=0$. For every $P\in \mathcal{P}$ there is an exact sequence 
	$$P\stackrel{f} \rightarrow D\twoheadrightarrow C,~D\in \mathcal{C} \cap \mathcal{D}, C\in \mathcal{C}$$ 
	with $f$ a left $\mathcal{D}$-approximation. Because $\mathcal{D}$ is factor closed, $C\in \mathcal{D}$. Thus $\mathcal{C} \cap \mathcal{D}$ is a support $\tau$-tilting subcategory.
	
	To prove the two maps are mutually inverse bijections, by Lemma \ref{enough proj T}, it suffices to show $\mathcal{D}=\mathsf{Fac}(\mathcal{C} \cap \mathcal{D})$. Let $X\in \mathcal{D}$ and $P\twoheadrightarrow X$ be a deflation with $P\in \mathcal{P}$. Then there is a commutative diagram
	\[\begin{tikzcd}
		P \arrow[r,"f"] \arrow[d,two heads] & D \arrow[dl,"g",dashed] \arrow[r,two heads] & C\\
		X&&
	\end{tikzcd}\]
	where the first row exists by definition of $\tau$-cotorsion pair and $g$ exists since $f$ is an approximation. Moreover $g$ is a deflation by Lemma \ref{deflation}, hence $X\in \mathsf{Fac}(\mathcal{C} \cap \mathcal{D})$. The inverse inclusion is obvious.
	
	Assume $\mathcal{T}$ is a tilting subcategory. For any $X\in \mathcal{E}$ we have a deflation $P\stackrel{a}\twoheadrightarrow X$ for some $P\in \mathcal{P}$ and then a conflation 
	$$P\rightarrowtail T^{0}\twoheadrightarrow T^{1},~T^{0},T^{1}\in \mathcal{T}.$$ 
	Consider the pushout diagram
	\[\begin{tikzcd}
		P \arrow[r,tail] \arrow[d,two heads,"a" swap] \arrow[dr, phantom, "{\rm PO}"]
		& T^{0} \arrow[r,two heads] \arrow[d,two heads] &T^{1} \arrow[d,equal]\\
		X \arrow[r,tail] & D \arrow[r,two heads] &T^{1},
	\end{tikzcd}\]
    we obtain a conflation $X\rightarrowtail D\twoheadrightarrow T^{1}$ where $D\in \mathsf{Fac}\mathcal{T}$ and $T^{1}\in \mathcal{T}$. Consider the pushout diagram
    \[\begin{tikzcd}
    	Y \arrow[r,tail] \arrow[d,tail] \arrow[dr, phantom, "{\rm PO}"]
    	& P \arrow[r,two heads,"a"] \arrow[d,tail] & X \arrow[d,equal] \\
    	D^{'} \arrow[r,tail] \arrow[d,two heads] & C \arrow[r,two heads] \arrow[d,two heads] & X \\
    	T^{'} \arrow[r,equal] & T^{'}, &
    \end{tikzcd}\]
    where the first column is obtained as above. Since $P,T^{'}\in {^{\bot_{1}}\mathsf{Fac}\mathcal{T}}$, then $C\in {^{\bot_{1}}\mathsf{Fac}\mathcal{T}}$. Thus the second row implies X admits a conflation with $D'\in \mathsf{Fac}\mathcal{T},~C\in {^{\bot_{1}}\mathsf{Fac}\mathcal{T}}$. Therefore $(^{\bot_{1}}\mathsf{Fac}\mathcal{T},\mathsf{Fac}\mathcal{T})$ is a cotorsion pair.

    If $(\mathcal{C},\mathcal{D})$ is a cotorsion pair such that $\mathcal{D}$ is a torsion class, by the second paragraph and Lemma \ref{criterion}, $\mathcal{C} \cap \mathcal{D}$ is a tilting subcategory.
\end{proof}

We call a support $\tau$-tilting subcategory $\mathcal{T}$ {\em admissibly contravariantly finite} (a.c.f. for short) in $\mathcal{E}$ if for every $X\in \mathcal{E}$ there is an admissible right $\mathcal{T}$-approximation.

\begin{cor}\label{direct generalization}
	The bijections in Theorem \ref{bijections} restrict to bijections 
	\begin{align*}
		\{\text{{\rm a.c.f.} support}~\tau\text{-tilting subcategories}\} & \leftrightarrow \{\tau\text{-cotorsion torsion triples}\}\\
		\mathcal{T} & \mapsto (^{\bot_{1}}\mathsf{Fac}\mathcal{T},\mathsf{Fac}\mathcal{T},\mathcal{T}^{\bot_{0}})\\
		\mathcal{C} \cap \mathcal{D} & \mapsfrom (\mathcal{C},\mathcal{D},\mathcal{F})
	\end{align*}
	and
	\[\{\text{{\rm a.c.f.} tilting subcategories}\}\leftrightarrow \{\text{cotorsion torsion triples}\}.\]
\end{cor}
\begin{proof}
	If $\mathcal{T}$ is an a.c.f. support $\tau$-tilting subcategory of $\mathcal{E}$, it suffices to check $(\mathsf{Fac}\mathcal{T},\mathcal{T}^{\bot_{0}})$ is a torsion pair. For each $X\in \mathcal{E}$, there is an admissible right $\mathcal{T}$-approximation $T\stackrel{f}\rightarrow X$. Thus we have a conflation 
	$${\rm Im}f\rightarrowtail X\twoheadrightarrow {\rm coker}f$$ 
	with ${\rm Im}f\in \mathsf{Fac}\mathcal{T}$. Apply ${\rm Hom}_{\mathcal{E}}(\mathcal{T},-)$ to it and we deduce ${\rm Hom}_{\mathcal{E}}(\mathcal{T},{\rm coker}f)=0$. Hence ${\rm coker}f\in \mathcal{T}^{\bot_{0}}$.
	
	Conversely, if $(\mathcal{C},\mathcal{D},\mathcal{F})$ is a $\tau$-cotorsion torsion triple, then $\mathcal{D}$ is a torsion class. Therefore there is a support $\tau$-tilting subcategory $\mathcal{T}$ of $\mathcal{E}$ such that $(\mathcal{C},\mathcal{D},\mathcal{F})=(^{\bot_{1}}\mathsf{Fac}\mathcal{T},\mathsf{Fac}\mathcal{T},\mathcal{T}^{\bot_{0}})$. Let $X\in \mathcal{E}$, because $(\mathsf{Fac}\mathcal{T},\mathcal{T}^{\bot_{0}})$ is a torsion pair, we have a conflation 
	\[Y\stackrel{i}\rightarrowtail X\twoheadrightarrow Z\]
	with $Y\in \mathsf{Fac}\mathcal{T},~Z\in \mathcal{T}^{\bot_{0}}$ and $i$ a right $\mathsf{Fac}\mathcal{T}$-approximation. By Lemma \ref{enough proj T}, for $Y\in \mathsf{Fac}\mathcal{T}$, we have a conflation 
	\[Y'\rightarrowtail T\stackrel{d}\twoheadrightarrow Y\]
	with $Y'\in \mathsf{Fac}\mathcal{T},~T\in \mathcal{T}$ and $d$ a right $\mathcal{T}$-approximation. Thus $i\circ d:T\rightarrow X$ is an admissible right $\mathcal{T}$-approximation and $\mathcal{C} \cap \mathcal{D}=\mathcal{T}$ is a.c.f.
	
	The second bijection is obvious.
\end{proof}

\begin{rem}
	From Corollary \ref{direct generalization}, we see that for any $\tau$-cotorsion torsion triple $(\mathcal{C},\mathcal{D},\mathcal{F})$, $\mathcal{C} \cap \mathcal{D}$ is a.c.f. Thus condition (3) of \cite[Definition 4.1]{AST} is redundant in a sense. If $\mathcal{E}$ is an abelian category, Corollary \ref{direct generalization} is just \cite[Theorem 5.7]{AST}. In a general exact category, however we do not know if the mid-term of a $\tau$-cotorsion torsion triple is functorially finite (it is true in abelian categories, see \cite[Corollary 4.8]{AST}).
\end{rem}

\section{Restrict to be tilting subcategories}

Throughout this section, we assume that $\mathcal{E}=(\mathcal{A},\mathcal{S})$ is a skeletally small exact category unless otherwise specified.

\subsection{Restrict to a subcategory}

In ${\rm mod}-\Lambda$ ($\Lambda$ is an artin algebra), for a support $\tau$-tilting module $T$, we can find a quotient algebra $\Lambda/I$ such that $T\in {\rm mod}-\Lambda/I$ is tilting. We want to construct a similar subcategory $\mathcal{E}_{\mathcal{T}}$ for a support $\tau$-tilting subcategory $\mathcal{T}$. 

Let $\mathcal{E}=(\mathcal{A},\mathcal{S})$ be an exact category with enough projectives $\mathcal{P}$, $\mathcal{T}$ be a $\tau$-rigid subcategory satisfying (A). Recall that we have a functor
\begin{align*}
 \mathbb{P}:~\mathcal{E}&\longrightarrow {\rm mod}_{\infty}-\mathcal{P} \subseteq {\rm Mod}-\mathcal{P}\\
 X&\longmapsto \mathcal{E}(-,X)|_{\mathcal{P}}
\end{align*}
which is fully faithful (see Lemma \ref{Enomoto}). Let $\mathcal{I}=\{f\in \mathcal{P}~|~\mathcal{E}(f,\mathcal{T})=0\}$, it is an ideal of the additive category $\mathcal{P}$. Denote by $\overline{\mathcal{P}}$ the quotient category $\mathcal{P}/\mathcal{I}$ and by $\pi$ the canonical quotient functor $\mathcal{P}\rightarrow \overline{\mathcal{P}}$. It is known that there is an adjoint pair (cf. for example \cite[Section 3]{Auslander})
\begin{align*}
-\otimes_{\mathcal{P}}\overline{\mathcal{P}}:{\rm Mod}-\mathcal{P}&\longrightarrow {\rm Mod}-\overline{\mathcal{P}}\\
X&\longmapsto (Q\mapsto X\otimes_{\mathcal{P}}\overline{\mathcal{P}}(Q,\pi(-)))\\
(-)\circ \pi:{\rm Mod}-\overline{\mathcal{P}}&\longrightarrow {\rm Mod}-\mathcal{P}\\
Y&\longmapsto Y\circ \pi.
\end{align*}
Let $\epsilon$ and $\eta$ denote the counit and unit respectively. Since $(-)\circ \pi$ is fully faithful and exact, $\epsilon$ is a natural isomorphism and $-\otimes_{\mathcal{P}}\overline{\mathcal{P}}$ is right exact and preserves projectives. Regard ${\rm Mod}-\overline{\mathcal{P}}$ as a subcategory of ${\rm Mod}-\mathcal{P}$, it is subobject closed and factor closed.

\begin{lem}\label{lem:restriction modP}
	${\rm Mod}-\overline{\mathcal{P}} \cap {\rm mod}_{\infty}-\mathcal{P}={\rm mod}_{\infty}-\overline{\mathcal{P}}$.
\end{lem}
\begin{proof}
	For any projective $Q\in \mathcal{P}$ 
	there exists an exact sequence 
	$$Q\stackrel{f} \rightarrow T\twoheadrightarrow C,~T\in \mathcal{T}$$ 
	with $f$ a left $\mathcal{T}$-approximation because $\mathcal{T}$ satisfies condition (A). So we obtain an exact sequence 
	$$\mathcal{P}(-,Q) \rightarrow (-,T)|_{\mathcal{P}} \rightarrow (-,C)|_{\mathcal{P}} \rightarrow 0$$ 
	in ${\rm Mod}-\mathcal{P}$. Applying $-\otimes_{\mathcal{P}}\overline{\mathcal{P}}$ and $(-)\circ \pi$ we obtain a commutative diagram with exact rows
	\[\begin{tikzcd}
		& \mathcal{P}(-,Q) \arrow[r] \arrow[d,"\eta_{Q}"] & (-,T)|_{\mathcal{P}} \arrow[r] \arrow[d,equal] & (-,C)|_{\mathcal{P}} \arrow[r] \arrow[d,equal] & 0 \\
		0 \arrow[r] & \overline{\mathcal{P}}(-,Q)\arrow[r,"i"] & (-,T)|_{\mathcal{P}} \arrow[r] & (-,C)|_{\mathcal{P}} \arrow[r] & 0 
	\end{tikzcd}\]
    where $\eta_{Q}$ is the canonical surjective map. One can check that $i$ is a monomorphism because $f$ is an approximation. 
    Denote by $K$ the image of $f$, then $\overline{\mathcal{P}}(-,Q)\cong (-,K)|_{\mathcal{P}}$ in ${\rm Mod}-\mathcal{P}$ and so $\overline{\mathcal{P}}(-,Q)\in {\rm Im}\mathbb{P}$. Hence ${\rm mod}_{\infty}-\overline{\mathcal{P}} \subseteq {\rm Mod}-\overline{\mathcal{P}} \cap {\rm mod}_{\infty}-\mathcal{P}$ by Lemma \ref{infinite resolution}.

    Assume $X\in {\rm Mod}-\overline{\mathcal{P}} \cap {\rm mod}_{\infty}-\mathcal{P}$, there exists an exact sequence
    \[\cdots \rightarrow \mathcal{P}(-,Q_{n})\rightarrow \cdots \rightarrow \mathcal{P}(-,Q_{1})\rightarrow \mathcal{P}(-,Q_{0})\rightarrow X \rightarrow 0.\]
    Applying $-\otimes_{\mathcal{P}}\overline{\mathcal{P}}$ and $(-)\circ \pi$ we obtain
    \[0 \rightarrow K_{1} \rightarrow \overline{\mathcal{P}}(-,Q_{1})\rightarrow \overline{\mathcal{P}}(-,Q_{0})\rightarrow X \rightarrow 0\]
    where $K_{1} \in {\rm Mod}-\overline{\mathcal{P}} \cap {\rm mod}_{\infty}-\mathcal{P}$ because ${\rm Mod}-\overline{\mathcal{P}}$ is subobject closed and ${\rm mod}_{\infty}-\mathcal{P}$ is closed under kernels of a deflation. Repeat the construction above, it follows that $X \in {\rm mod}_{\infty}-\overline{\mathcal{P}}$.
\end{proof}

Define 
$$\mathcal{A}_{\mathcal{T}}:=\{X\in \mathcal{E}~|~\mathbb{P}(X)\in {\rm mod}_{\infty}-\overline{\mathcal{P}}\}.$$ 
Since $\mathbb{P}$ is fully faithful and by Lemma~\ref{lem:restriction modP} we have $\mathcal{A}_{\mathcal{T}}=\{X\in \mathcal{E}~|~\mathcal{E}(\mathcal{I},X)=0\}.$
Moreover,
$\mathcal{A}_{\mathcal{T}}$ is a subobject, factor and additively closed subcategory. Set
\begin{align*}
\mathcal{S}_{\mathcal{T}}&:=\{\text{conflations in}~\mathcal{S}~\text{such that all three terms lie in}~\mathcal{A}_{\mathcal{T}}\},\\
\mathcal{E}_{\mathcal{T}}&:=(\mathcal{A}_{\mathcal{T}}, \mathcal{S}_{\mathcal{T}}).  
\end{align*}

\begin{lem}\label{find subcategory}
	Assume $\mathcal{E}$ is weakly idempotent complete. Then $\mathcal{E}_{\mathcal{T}}$ is an exact category with enough projectives given by\\
	$\mathcal{P}_{\mathcal{T}}=\mathsf{add} \{K~|~\exists~Q \in \mathcal{P}~{\rm and~an~admissible~left~\mathcal{T}-approximation}~Q\stackrel{f} \rightarrow T~{\rm s.t.}~K={\rm Im}f \}.$
\end{lem}
\begin{proof}
	Firstly we prove $\mathcal{E}_{\mathcal{T}}$ is an exact category. Let $X\rightarrowtail Y\twoheadrightarrow Z$ be a conflation in $\mathcal{S}_{\mathcal{T}}$ and $g:X\rightarrow M$ be a morphism in $\mathcal{E}_{\mathcal{T}}$. Consider the pushout diagram in $\mathcal{E}$
	\[\begin{tikzcd}
		X \arrow[r,tail] \arrow[d,"g"] \arrow[dr, phantom, "{\rm PO}"] & Y \arrow[r,two heads] \arrow[d] & Z \arrow[d,equal]\\
		M \arrow[r,tail] & N \arrow[r,two heads] & Z,
	\end{tikzcd}\]
    it induces a conflation $X \rightarrowtail M\oplus Y \twoheadrightarrow N$ in $\mathcal{E}$ (see \cite[Proposition 2.12]{T.Buhler}). Since  $\mathcal{E}_{\mathcal{T}}$ is factor closed and $M\oplus N\in  \mathcal{E}_{\mathcal{T}}$, then $N\in \mathcal{E}_{\mathcal{T}}$. Thus the pushout axiom holds for $\mathcal{E}_{\mathcal{T}}$. Similarly the pullback axiom also holds. Therefore $\mathcal{E}_{\mathcal{T}}$ is an exact category.

    Let $Q \in \mathcal{P}$, $Q\stackrel{f} \rightarrow T\twoheadrightarrow C$ be an exact sequence with $f$ is a left $\mathcal{T}$-approximation and $K={\rm Im}f$. Then by the proof of Lemma \ref{lem:restriction modP}, we have $\overline{\mathcal{P}}(-,Q)\cong (-,K)|_{\mathcal{P}}$. For any conflation $M \rightarrowtail N \twoheadrightarrow K$ in $\mathcal{S}_{\mathcal{T}}$, its image under $\mathbb{P}$ is an exact sequence 
    $$0 \rightarrow (-,M)|_{\mathcal{P}} \rightarrow (-,N)|_{\mathcal{P}} \rightarrow (-,K)|_{\mathcal{P}}\rightarrow 0$$ 
    in ${\rm Mod}-\overline{\mathcal{P}}$ (by the definition of $\mathcal{A}_{\mathcal{T}}$) which splits since $(-,K)|_{\mathcal{P}}$ is projective. Because $\mathbb{P}$ is fully faithful, the conflation $M \rightarrowtail N \twoheadrightarrow K$ also splits. This shows $K\in \mathcal{E}_{\mathcal{T}}$ is projective.

    For any $X\in \mathcal{E}_{\mathcal{T}}$ we have a deflation $Q\stackrel{a}\twoheadrightarrow X$ for some $Q\in \mathcal{P}$ and then an exact sequence $Q\stackrel{f} \rightarrow T\twoheadrightarrow C$ as above. The commutative diagram
    \[\begin{tikzcd}[column sep=small,row sep=tiny]
    	Q \arrow[rr,"f"] \arrow[dr,two heads,"j" swap] \arrow[dd,two heads,"a" swap] & & T \arrow[r,two heads] & C \\
    	& K \arrow[ur,tail,"i" swap] & & \\
    	X&&&
    \end{tikzcd}\] induces the following commutative diagram by applying $\mathbb{P}$
    \[\begin{tikzcd}[column sep=small,row sep=tiny]
    	\mathcal{P}(-,Q) \arrow[rr] \arrow[dr] \arrow[dd] & & (-,T)|_{\mathcal{P}} \arrow[r] & (-,C)|_{\mathcal{P}} \arrow[r] & 0 \\
    	& (-,K)|_{\mathcal{P}} \arrow[ur] \arrow[dl,dashed] & & & \\
    	(-,X)|_{\mathcal{P}} & & & & 
    \end{tikzcd}.\]
    Then apply $-\otimes_{\mathcal{P}}\overline{\mathcal{P}}$ and $(-)\circ \pi$, we obtain a morphism $(-,K)|_{\mathcal{P}} \rightarrow (-,X)|_{\mathcal{P}}$ making the above diagram commute.
    By fully faithfulness of $\mathbb{P}$, there is a morphism $K\stackrel{d}\rightarrow X$ such that $dj=a$, hence it is a deflation by Lemma \ref{deflation}. Thus the assertion follows.
\end{proof}

\begin{cor}
	If $\mathcal{E}$ is abelian, so is $\mathcal{E}_{\mathcal{T}}$.
\end{cor}
\begin{proof}
	It follows by the fact that $\mathcal{E}_{\mathcal{T}}$ is factor and subobject closed and every morphism in $\mathcal{E}$ is admissible. 
\end{proof}

In general, $\mathcal{E}_{\mathcal{T}}$ is not a full exact subcategory of $\mathcal{E}$ (see Remark \ref{rem:rest subcat}(3)). Now we state our first main result of this section. 

\begin{thm}\label{thm2}
	Assume $\mathcal{E}$ is weakly idempotent complete. If $\mathcal{T}$ is a $\tau$-rigid subcategory satisfying (A). Then $\mathcal{T}$ is a partial tilting subcategory of $\mathcal{E}_{\mathcal{T}}$. In particular, if $\mathcal{T}$ is a support $\tau$-tilting subcategory, then $\mathcal{T}$ is a tilting subcategory of $\mathcal{E}_{\mathcal{T}}$.
\end{thm}
\begin{proof}
	Clearly $\mathcal{T}\subseteq \mathcal{E}_{\mathcal{T}}$ and ${\mathsf{Fac}}_{\mathcal{E}}\mathcal{T}={\mathsf{Fac}}_{\mathcal{E}_{\mathcal{T}}}\mathcal{T}$. Because ${\rm Ext}_{\mathcal{E}}^{1}(\mathcal{T},\mathsf{Fac}\mathcal{T})=0$, we have ${\rm Ext}_{\mathcal{E}_{\mathcal{T}}}^{1}(\mathcal{T},\mathsf{Fac}\mathcal{T})=0$. By the same proof as Lemma \ref{criterion}, we can show that ${\rm pd}_{\mathcal{E}_{\mathcal{T}}}\mathcal{T}\leq 1$. Indeed, it suffices to show that for every $X\in \mathcal{E}_{\mathcal{T}}$ there is a deflation $K\twoheadrightarrow X$ for some $K\in \mathcal{P}_{\mathcal{T}}$ and a conflation 
	\[K\rightarrowtail T\twoheadrightarrow C\]
	with $T\in \mathcal{T}$. By the last paragraph of the proof of Lemma \ref{find subcategory}, it is obvious. 
	
	Assume $\mathcal{T}$ is a support $\tau$-tilting subcategory. If $K\in \mathcal{P}_{\mathcal{T}}$ such that $K={\rm Im}f$ as in Lemma \ref{find subcategory}, by the proof of Proposition \ref{inclusion}, there is a conflation $K \rightarrowtail T^{0} \twoheadrightarrow T^{1}$ such that $T^{0},T^{1}\in \mathcal{T}$. Let $Q'\in \mathcal{P}_{\mathcal{T}}$ be any projective object, then it is a summand of some $K={\rm Im}f$. Consider the pushout diagram
	\[\begin{tikzcd}
		K \arrow[r,tail] \arrow[d,two heads] \arrow[dr, phantom, "{\rm PO}"] & T^{0} \arrow[r,two heads] \arrow[d,two heads] &T^{1} \arrow[d,equal]\\
		Q' \arrow[r,tail] & T'^{0} \arrow[r,two heads] & T^{1}.
	\end{tikzcd}\]
    Apply ${\rm Hom}_{\mathcal{E}_{\mathcal{T}}}(-,\mathsf{Fac}\mathcal{T})$ to the second row and we have ${\rm Ext}_{\mathcal{E}_{\mathcal{T}}}^{1}(T'^{0},\mathsf{Fac}\mathcal{T})=0$. Because $T'^{0}\in \mathsf{Fac}\mathcal{T}$ and the exact structures of $\mathsf{Fac}\mathcal{T}$ as full exact subcategories of $\mathcal{E}_{\mathcal{T}}$ and $\mathcal{E}$ coincides, we have $T'^{0}\in \mathcal{T}$. Thus $\mathcal{T}$ is a tilting subcategory of $\mathcal{E}$.
\end{proof}

\begin{rem}\label{rem:rest subcat}
	(1) We do not need the assumption that $\mathcal{T}$ is $\tau$-rigid until Theorem \ref{thm2}, i.e. the previous lemmas in the section hold for a subcategory $\mathcal{T}$ satisfying (A).
	
	(2) If $\mathcal{T}$ is a tilting subcategory, then $\mathcal{I}=0$ and $\mathcal{E}_{\mathcal{T}}=\mathcal{E}$.
	
	(3) If $\mathcal{E}={\rm mod}-\Lambda$ ($\Lambda$ is an artin algebra), $\mathcal{T}=\mathsf{add}T$ for a $\tau$-rigid (resp. support $\tau$-tilting) module $T$. Then $\mathcal{E}_{\mathcal{T}}= {\rm mod}-\Lambda/{\rm ann}T$ and $T\in \mathcal{E}_{\mathcal{T}}$ is a partial tilting (resp. tilting) module. Indeed, by Lemma \ref{find subcategory} we have $\mathcal{E}_{\mathcal{T}}= \mathsf{Fac}\mathcal{P}_{\mathcal{T}}$. Here $\mathcal{P}_{\mathcal{T}}=\mathsf{add}(\Lambda/{\rm ann}T)$ by Lemma \ref{lem:KS-proj} and thus $\mathcal{E}_{\mathcal{T}}=\mathsf{Fac}(\mathsf{add}(\Lambda/{\rm ann}T))={\rm mod}-\Lambda/{\rm ann}T$.
\end{rem}

\subsection{Number of indecomposables}

If $\mathcal{E}$ is Krull-Schmidt, for an additively closed subcategory $\mathcal{T}$, denote by $|\mathcal{T}|$ the cardinal of the set of isomorphism classes of indecomposable objects in $\mathcal{T}$. For an object $T$, $|T|=|\mathsf{add}T|$. The following result generalizes \cite[Theorem 5.11]{Sauter} and coincides with the well-known fact in classical tilting theory.

\begin{prop}\label{lem:number components}
	Let $\mathcal{E}=(\mathcal{A},\mathcal{S})$ be a Krull-Schmidt exact category. Then for any $n$-tilting subcategory $\mathcal{T}$, $|\mathcal{T}|$ is a definite number.
\end{prop}
\begin{proof}
	Define $\mathcal{P}^{<\infty}:=\{X\in \mathcal{E}~|~\exists~n\geq 0~\text{s.t.}~{\rm Ext}_{\mathcal{E}}^{n+1}(X,-)=0\}$. It is a thick subcategory of $\mathcal{E}$. It follows \cite[Lemma 4.8]{Sauter} that $\mathcal{T}$ is an $n$-tilting subcategory of $\mathcal{P}^{<\infty}$ and ${\rm Thick}_{\mathcal{P}^{<\infty}}(\mathcal{T})=\mathcal{P}^{<\infty}$. Thus by \cite[Lemma 4.7]{Sauter}, we have a triangle equivalence 
	$$K^{b}(\mathcal{T}) \stackrel{\simeq}\longrightarrow D^{b}(\mathcal{P}^{<\infty}).$$ 
	Then by \cite[Lemma 4.1.17]{H.Krause} we have $$K_{0}(\mathcal{P}^{<\infty})\stackrel{\simeq}\longrightarrow K_{0}(D^{b}(\mathcal{P}^{<\infty}))\stackrel{\simeq}\longrightarrow K_{0}(K^{b}(\mathcal{T}))\stackrel{\simeq}\longrightarrow K_{0}(\mathcal{T})$$ 
	where the split Grothendieck group $K_{0}(\mathcal{T})$ is a free abelian group with a basis consisting of isomorphism classes of indecomposable objects in $\mathcal{T}$. Thus the assertion follows.
\end{proof}

Assume for the rest of the subsection that $\mathcal{E}=(\mathcal{A},\mathcal{S})$ is a Krull-Schmidt exact category with enough projectives $\mathcal{P}$ and $\mathcal{T}$ is a $\tau$-rigid subcategory satisfying (A). 

Now we can prove another main result concerning the number of indecomposable objects in a support $\tau$-tilting subcategory. Recall that if an additive category is Krull-Schmidt, for a morphism $g:X\rightarrow Y$, by \cite{Bian Ning} we can always decompose $X$ (resp. $Y$) to obtain a right (resp. left) minimal morphism $g'$.

For each $Q\in \mathcal{P}$, by condition (A), there is an exact sequence 
$$Q\stackrel{f} \rightarrow T\twoheadrightarrow C$$ 
with $T\in \mathcal{T}$ and $f$ a left $\mathcal{T}$-approximation. We can decompose $T$ to obtain an admissible minimal left approximation $f':Q\rightarrow T'$ by Lemma \ref{admissible}. Thus for every $Q$, $K={\rm Im}f$ does not depend on the choice of $f$. We fix such an exact sequence with left minimal $f$ for each $Q\in \mathcal{P}$ and define a functor $\rho:\mathcal{P}\rightarrow \mathcal{P}_{\mathcal{T}}$ as follows:

(1) For an object $Q\in \mathcal{P}$, $\rho(Q)=K={\rm Im}f$.

(2) For a morphism $a:Q\rightarrow Q'$, consider the commutative diagram
\[\begin{tikzcd}[column sep=small,row sep=tiny]
	Q \arrow[rr,"f"] \arrow[dr,two heads] \arrow[dd,"a"] & & T \arrow[r,two heads,"d"] \arrow[dd,"b",dashed] & C \arrow[dd,"c",dashed] \\
	& K \arrow[ur,tail,"i" swap] \arrow[dd,"\tilde{a}",near start,dashed] & & \\
	Q' \arrow[rr,"f'",near start] \arrow[dr,two heads] & & T' \arrow[r,two heads,"d'"] & C' \\
	& K' \arrow[ur,tail,"i'",swap] & &
\end{tikzcd}.\] We set $\rho(a)=\tilde{a}$.

\begin{lem}\label{lem:KS-proj}
	The functor $\rho$ is well-defined and induces an equivalence $\overline{\rho}:\overline{\mathcal{P}}\rightarrow \mathcal{P}_{\mathcal{T}}$.
\end{lem}
\begin{proof}
	To show $\rho$ is well-defined it suffices to show $\tilde{a}$ does not depend on the choice of $b$. If there is another triple ($b',c',\tilde{a'}$) making the diagram commute. Then $(b-b')f=0$, hence $b-b'$ factor through $d$. Therefore $i'(\tilde{a}-\tilde{a'})=0$ and $\tilde{a}=\tilde{a'}$.
    
    Obviously, $\rho$ is additive and full. If $\tilde{a}=0$, then $f'a=0$. Since $f'$ is an approximation, $a\in \mathcal{I}$. Conversely if $a\in \mathcal{I}$, then $\tilde{a}=0$. Thus $\rho$ factor through $\pi$ and induces a fully faithful $\overline{\rho}:\overline{\mathcal{P}}\rightarrow \mathcal{P}_{\mathcal{T}}$.
    
    If $Q\in \mathcal{P}$ is indecomposable, then $K=\rho(Q)$ is also indecomposable because ${\rm End}_{\mathcal{E}}(K)\cong {\rm End}_{\mathcal{E}}(Q)/\mathcal{I}(Q,Q)$ which is local. Thus $\overline{\rho}$ is an equivalence.
\end{proof}

\begin{thm}\label{thm:number component}
	Let $\mathcal{T}$ be a support $\tau$-tilting subcategory of $\mathcal{E}$. Then $|\mathcal{T}|$ equals to the number of isomorphism classes of indecomposable projectives $Q$ such that ${\rm Hom}_{\mathcal{E}}(Q,\mathcal{T})\neq 0$.
\end{thm}
\begin{proof}
	Let $\mathcal{Q}$ denote the additively closed subcategory of which the indecomposable objects are indecomposable projectives $Q$ such that ${\rm Hom}_{\mathcal{E}}(Q,\mathcal{T})\neq 0$. We claim $|\mathcal{Q}|=|\mathcal{P}_{\mathcal{T}}|$. By Lemma \ref{lem:KS-proj}. It suffices to show that if indecomposable projective objects $Q_{1},Q_{2}\in \mathcal{Q}$ are isomorphic in $\overline{\mathcal{P}}$, they are also isomorphic in $\mathcal{P}$. Indeed, assume $f:Q_{1}\rightarrow Q_{2},~g:Q_{2}\rightarrow Q_{1}$ such that $\overline{fg}=\overline{\rm 1},~\overline{gf}=\overline{\rm 1}$. Then $1-gf\in {\rm rad}({\rm End}_{\mathcal{E}}(Q_{1}))$ hence $gf$ is an isomorphism. Similarily, we deduce $fg$ is an isomorphism and therefore $Q_{1}\cong Q_{2}$ in $\mathcal{P}$.
	
	By Theorem \ref{thm2} and Proposition~\ref{lem:number components} we have $|\mathcal{Q}|=|\mathcal{P}_{\mathcal{T}}|=|\mathcal{T}|$.
\end{proof}

\begin{cor}
	If there exists a $\tau$-tilting object $T$. Then $|\mathcal{P}|=|T|<\infty$ and every $\tau$-rigid subcategory satisfying (A) is of the form $\mathsf{add}X$ for some $X$.
\end{cor}
\begin{proof}
	By Theorem \ref{thm:number component} and Proposition \ref{inclusion}, it is clear.
\end{proof}

\begin{cor}\label{Cor:number of ind}
	Assume $|\mathcal{P}|<\infty$. For a $\tau$-rigid object $T$ satisfying (A), consider a pair $(T,Q)$ such that $Q\in \mathcal{P},~{\rm Hom}_{\mathcal{E}}(Q,T)=0$ and $|Q|$ is maximal. Then $|T|+|Q|\leq |\mathcal{P}|$ and it is an equality if and only if $T$ is a support $\tau$-tilting object.
\end{cor}
\begin{proof}
	$|Q|+|{\rm P}(\mathsf{Fac}T)|=|\mathcal{P}|$ and the assertion follows by Corollary \ref{T=Ext-proj}.
\end{proof}

For convenience, we call the pair $(T,Q)$ in Corollary \ref{Cor:number of ind} a $\tau$-rigid (resp. support $\tau$-tilting) pair if $T$ is $\tau$-rigid (resp. support $\tau$-tilting) object.

\section{Generalized Brenner-Butler's theorem}

The following theorem generalizes \cite[Proposition 3.5]{G.Jasso} and partially generalizes \cite[Theorem 6.13]{Sauter} (cf. Remark \ref{final remark}). 

\begin{thm}\label{thm4}
	Let $\mathcal{E}=(\mathcal{A},\mathcal{S})$ be a skeletally small idempotent complete exact category with enough projectives $\mathcal{P}$ and $\mathcal{T}$ be a support $\tau$-tilting subcategory. Consider the functor
	\[\mathbb{P}:~\mathcal{E}\longrightarrow {\rm Mod}-\mathcal{P},~X\mapsto (-,X)|_{\mathcal{P}}\]
	and
	\[\mathbb{T}:~\mathcal{E}\longrightarrow \mathcal{T}-{\rm Mod},~X\mapsto (X,-)|_{\mathcal{T}}.\]
	Then we have:
	\begin{enumerate}
		\item $\mathsf{add}(\mathcal{P},-)|_{\mathcal{T}}$ is a tilting subcategory of $\mathcal{T}-{\rm mod}_{\infty}$.
		\item $(-,\mathcal{T})|_{\mathcal{P}}$ is a tilting subcategory of ${\rm mod}_{\infty}-\overline{\mathcal{P}}$.
		\item The adjoint pair
		\[-\otimes_{\mathcal{T}}\Psi:{\rm Mod}-\mathcal{T}\longrightarrow {\rm Mod}-\mathcal{P},~X\mapsto (Q\mapsto X\otimes_{\mathcal{T}}(Q,-)|_{\mathcal{T}})\]
		and
		\[{\rm Hom}_{\mathcal{P}}(\Psi,-):{\rm Mod}-\mathcal{P}\longrightarrow {\rm Mod}-\mathcal{T},~Y\mapsto (T\mapsto {\rm Hom}_{\mathcal{P}}((-,T)|_{\mathcal{P}},Y))\]
		given by the bifunctor $\Psi:\mathcal{P}^{op} \times \mathcal{T} \rightarrow {\rm Ab},~(Q,T)\mapsto {\rm Hom}_{\mathcal{E}}(Q,T)$
		restrict to mutually inverse exact equivalences between
		\[\mathsf{Fac}(-,\mathcal{T})|_{\mathcal{P}}\subseteq {\rm mod}_{\infty}-\mathcal{P}~({\rm i.e.~factor~objects~of~(-,\mathcal{T})|_{\mathcal{P}}~in}~{\rm mod}_{\infty}-\mathcal{P})\]
		and
		\[{_{\bot}((\mathcal{P},-)|_{\mathcal{T}})}:=\{N\in {\rm mod}_{\infty}-\mathcal{T}~|~{\rm Tor}_{>0}^{\mathcal{T}}(N,(\mathcal{P},-)|_{\mathcal{T}})=0\}\]
		whose exact structures are obtained by restricting those of ${\rm mod}_{\infty}-\mathcal{P}$ and ${\rm mod}_{\infty}-\mathcal{T}$ respectively.
		\item There is a commutative triangle of exact functors
		\[\begin{tikzcd}[column sep=tiny]
			& \mathsf{Fac}\mathcal{T} \arrow[dl,"\mathbb{P}" swap,end anchor={[xshift=3ex]}] \arrow[dr,"\mathbb{T}'"] & \\
			{\rm mod}_{\infty}-\mathcal{P}\supseteq \mathsf{Fac}(-,\mathcal{T})|_{\mathcal{P}} \arrow[rr,"\simeq","{\rm Hom}_{\mathcal{P}}(\Psi{,}-)" swap] & & {_{\bot}((\mathcal{P},-)|_{\mathcal{T}})}
		\end{tikzcd}\]
	    where $\mathbb{T}':~\mathcal{E}\longrightarrow {\rm Mod}-\mathcal{T},~X\mapsto (-,X)|_{\mathcal{T}}$.
	\end{enumerate}

\end{thm}

We begin with some preparation. Let $\mathcal{T}$ be a $\tau$-rigid subcategory satisfying (A). Note that we can still define the functor $\rho:\mathcal{P}\rightarrow \mathcal{P}_{\mathcal{T}}$ as in the previous subsection without Krull-Schmidt assumption. But here, for every $Q\in \mathcal{P}$, we fix an arbitrary exact sequence $Q\stackrel{f} \rightarrow T\twoheadrightarrow C$ with $f$ a left $\mathcal{T}$-approximation instead. Obviously $\rho$ is still well-defined and induces a fully faithful functor $\overline{\rho}:\overline{\mathcal{P}}\rightarrow \mathcal{P}_{\mathcal{T}}$. But here $\overline{\rho}$ is not dense in general.

\begin{lem}
	The functor $\rho$ does not depend on the choice of $Q\stackrel{f} \rightarrow T\twoheadrightarrow C$ for every $Q\in \mathcal{P}$.
\end{lem}
\begin{proof}
	If we fix another sequence $Q\stackrel{f'} \rightarrow T'\twoheadrightarrow C'$ for every $Q$, then we obtain another $\rho'$. Consider the diagrams
	\[\begin{tikzcd}[column sep=small,row sep=tiny]
		Q \arrow[rr,"f"] \arrow[dr,two heads] \arrow[dd,equal] & & T \arrow[r,two heads] \arrow[dd,dashed] & C \arrow[dd,dashed] \\
		& K \arrow[ur,tail] \arrow[dd,"h",near start,dashed] & & \\
		Q \arrow[rr,"f'",near start] \arrow[dr,two heads] & & T' \arrow[r,two heads] & C' \\
		& K' \arrow[ur,tail] & &
	\end{tikzcd} {\rm and}~
      \begin{tikzcd}[column sep=small,row sep=tiny]
      	Q \arrow[rr,"f'"] \arrow[dr,two heads] \arrow[dd,equal] & & T' \arrow[r,two heads] \arrow[dd,dashed] & C' \arrow[dd,dashed] \\
      	& K' \arrow[ur,tail] \arrow[dd,"h'",near start,dashed] & & \\
      	Q \arrow[rr,"f",near start] \arrow[dr,two heads] & & T \arrow[r,two heads] & C \\
      	& K \arrow[ur,tail] & &
      \end{tikzcd}\]
    Then as in the proof of Lemma \ref{lem:KS-proj} we have $hh'=1$ and $h'h=1$. Thus $K\cong K'$ and $\rho \cong \rho'$ canonically. 
\end{proof}

Consider the functor
\[F:~{\rm Mod}-\mathcal{P}_{\mathcal{T}}\longrightarrow {\rm Mod}-\overline{\mathcal{P}},~X\mapsto X\circ \overline{\rho}\]
which is clearly exact.

\begin{lem}\label{triangle}
	The functor $F$ induces an exact equivalence 
	$${\rm mod}_{\infty}-\mathcal{P}_{\mathcal{T}}\stackrel{\simeq}\longrightarrow {\rm mod}_{\infty}-\overline{\mathcal{P}}$$ 
	and there is a commutative triangle of exact functors
	\[\begin{tikzcd}[column sep=tiny]
		& \mathcal{E}_{\mathcal{T}} \arrow[dl,"\mathbb{P}_{\mathcal{T}}",swap] \arrow[dr,"\mathbb{P}"] & \\
		{\rm mod}_{\infty}-\mathcal{P}_{\mathcal{T}} \arrow{rr}{\simeq}[swap]{F} & & {\rm mod}_{\infty}-\overline{\mathcal{P}}
	\end{tikzcd}\]
\end{lem}
\begin{proof}
	For any $X\in \mathcal{E}_{\mathcal{T}}$, we have $F((-,X)|_{\mathcal{P}_{\mathcal{T}}}) = (-,X)|_{\mathcal{P}}$. Indeed for every $Q\in \mathcal{P}$,
	$$(-,X)|_{\mathcal{P}_{\mathcal{T}}}\circ \overline{\rho}(Q)=(K,X) \cong (Q,X)=(-,X)|_{\mathcal{P}}(Q).$$ 
	If $X\in \mathcal{P}_{\mathcal{T}}$ then $X$ is a direct summand of some $K$ as in Lemma \ref{find subcategory}. Because $(-,K)|_{\mathcal{P}} \cong \overline{\mathcal{P}}(-,Q)$, the image $(-,X)|_{\mathcal{P}}$ of $(-,X)|_{\mathcal{P}_{\mathcal{T}}}=\mathcal{P}_{\mathcal{T}}(-,X)$ is a summand of $\overline{\mathcal{P}}(-,Q)$. Thus we have $F({\rm mod}_{\infty}-\mathcal{P}_{\mathcal{T}})\subseteq {\rm mod}_{\infty}-\overline{\mathcal{P}}$. Clearly $F$ is fully faithful on ${\rm Im}\mathbb{P}_{\mathcal{T}}$, so is it on ${\rm mod}_{\infty}-\mathcal{P}_{\mathcal{T}}$. We can easily check it is also dense and reflects exactness.
\end{proof}

\begin{proof}[Proof of Theorem 5.1]
	By Proposition \ref{thm2} we can construct a subcategory $\mathcal{E}_{\mathcal{T}}\subseteq \mathcal{E}$ such that $\mathcal{E}_{\mathcal{T}}$ is an exact category with enough projectives $\mathcal{P}_{\mathcal{T}}$ and $\mathcal{T}$ is a tilting subcategory of $\mathcal{E}_{\mathcal{T}}$. By \cite[Theorem 6.13(1)]{Sauter} we have $(-,\mathcal{T})|_{\mathcal{P}_{\mathcal{T}}}$ is a tilting subcategory of ${\rm mod}_{\infty}-\mathcal{P}_{\mathcal{T}}$ and $(\mathcal{P}_{\mathcal{T}},-)|_{\mathcal{T}}$ is a tilting subcategory of $\mathcal{T}-{\rm mod}_{\infty}$. Therefore (2) follows by Lemma \ref{triangle}. Because $\mathsf{add}(\mathcal{P},-)|_{\mathcal{T}}=(\mathcal{P}_{\mathcal{T}},-)|_{\mathcal{T}}$, then we obtain (1).
	
	By \cite[Theorem 6.13(4)]{Sauter} the adjoint pair
	\[-\otimes_{\mathcal{T}}\Psi_{\mathcal{T}}:{\rm Mod}-\mathcal{T}\longrightarrow {\rm Mod}-\mathcal{P}_{\mathcal{T}},~X\mapsto (K\mapsto X\otimes_{\mathcal{T}}(K,-)|_{\mathcal{T}})\] and
	\[{\rm Hom}_{\mathcal{P}_{\mathcal{T}}}(\Psi_{\mathcal{T}},-):{\rm Mod}-\mathcal{P}_{\mathcal{T}}\longrightarrow {\rm Mod}-\mathcal{T},~Y\mapsto (T\mapsto {\rm Hom}_{\mathcal{P}_{\mathcal{T}}}((-,T)|_{\mathcal{P}_{\mathcal{T}}},Y))\]
	 restrict to mutually inverse exact equivalences between
	\[\{M\in {\rm mod}_{\infty}-\mathcal{P}_{\mathcal{T}}~|~{\rm Ext}_{{\rm Mod}-\mathcal{P}_{\mathcal{T}}}^{>0}((-,\mathcal{T})|_{\mathcal{P}_{\mathcal{T}}},M)=0\}=\mathsf{Fac}(-,\mathcal{T})|_{\mathcal{P}_{\mathcal{T}}}\subseteq {\rm mod}_{\infty}-\mathcal{P}_{\mathcal{T}}\] and
	\[\{N\in {\rm mod}_{\infty}-\mathcal{T}~|~{\rm Tor}_{>0}^{\mathcal{T}}(N,(\mathcal{P}_{\mathcal{T}},-)|_{\mathcal{T}})=0\}={_{\bot}((\mathcal{P},-)|_{\mathcal{T}})}\]
	where 
	$$\Psi_{\mathcal{T}}:\mathcal{P}_{\mathcal{T}}^{op} \times \mathcal{T} \rightarrow {\rm Ab},~(K,T)\mapsto {\rm Hom}_{\mathcal{E}_{\mathcal{T}}}(K,T)$$ 
	is a bifunctor. Note that
	$\mathsf{Fac}_{{\rm mod}_{\infty}-\overline{\mathcal{P}}}(-,\mathcal{T})|_{\mathcal{P}}=\mathsf{Fac}_{{\rm mod}_{\infty}-\mathcal{P}}(-,\mathcal{T})|_{\mathcal{P}}$. Thus the latter one has an exact structure. Moreover, we have a commutative diagram
	\[\begin{tikzcd}[row sep=huge,column sep=large]
		{\rm Mod}-\mathcal{P}_{\mathcal{T}} \arrow[dr,shift right=0.5ex,"{\rm Hom}(\Psi_{\mathcal{T}}{,}-)",swap,bend right] \arrow[r,"F"] & {\rm Mod}-\overline{\mathcal{P}} \arrow[d,shift right=0.5ex,"{\rm Hom}(\overline{\Psi}{,}-)",swap] \arrow[r,hook,"(-)\circ \pi"] & {\rm Mod}-\mathcal{P} \arrow[dl,shift right=0.5ex,bend left,"{\rm Hom}(\Psi{,}-)",swap,near start] \\
		& {\rm Mod}-\mathcal{T} \arrow[ul,shift right=0.5ex,"-\otimes_{\mathcal{T}}\Psi_{\mathcal{T}}",swap,bend left,near end] \arrow[u,shift right=0.5ex,"-\otimes_{\mathcal{T}}\overline{\Psi}",swap] \arrow[ur,shift right=0.5ex,bend right,"-\otimes_{\mathcal{T}}\Psi",swap] &
	\end{tikzcd}\]
    where 
    $$\overline{\Psi}:\overline{\mathcal{P}}^{op} \times \mathcal{T} \longrightarrow {\rm Ab},~(Q,T)\mapsto {\rm Hom}_{\mathcal{E}}(Q,T)$$ 
    is a bifunctor. This implies (3).
    
    By \cite[Theorem 6.13(5)]{Sauter} we have a commutative triangle of exact functors 
    \[\begin{tikzcd}[column sep=tiny]
    	& \mathsf{Fac}\mathcal{T} \arrow[dl,"\mathbb{P}_{\mathcal{T}}" swap] \arrow[dr,"\mathbb{T}'"] & \\
    	\mathsf{Fac}(-,\mathcal{T})|_{\mathcal{P}_{\mathcal{T}}} \arrow[rr,"\simeq","{\rm Hom}_{\mathcal{P}}(\Psi_{\mathcal{T}}{,}-)" swap] & & {_{\bot}((\mathcal{P},-)|_{\mathcal{T}})}
    \end{tikzcd}\]
    Combining with Lemma \ref{triangle}, then (4) follows.
\end{proof}

When $\mathcal{P}=\mathsf{add}P,~\mathcal{T}=\mathsf{add}T$ for some objects $P,T$. We have the following object version of Theorem \ref{thm4}.

\begin{cor}\label{object version}
		Let $\mathcal{E}=(\mathcal{A},\mathcal{S})$ be a skeletally small idempotent complete exact category with enough projectives $\mathcal{P}=\mathsf{add}P$ and $\mathcal{T}=\mathsf{add}T$ be a support $\tau$-tilting subcategory. Set $A={\rm End}_{\mathcal{E}}(P)$, $B={\rm End}_{\mathcal{E}}(T)$ and $M={\rm Hom}_{\mathcal{E}}(P,T)$. Then we have:
		\begin{enumerate}
			\item $M$ is a tilting object of $B-{\rm mod}_{\infty}$ and ${\rm mod}_{\infty}-A/I$, where $I={\rm ann}(M_{A})$.
			\item The adjoint pair
			\[-\otimes_{B}M:{\rm Mod}-B\longrightarrow {\rm Mod}-A\]
			and
			\[{\rm Hom}_{A}(M,-):{\rm Mod}-A\longrightarrow {\rm Mod}-B\]
			restrict to mutually inverse exact equivalences between
			\[\mathsf{Fac}M\subseteq {\rm mod}_{\infty}-A~({\rm i.e.~factor~objects~in}~{\rm mod}_{\infty}-A)\]
			and
			\[{_{\bot}M}:=\{N\in {\rm mod}_{\infty}-B~|~{\rm Tor}_{>0}^{B}(N,M)=0\}\]
			whose exact structures are obtained by restricting those of ${\rm mod}_{\infty}-A$ and ${\rm mod}_{\infty}-B$ respectively.
			\item There is a commutative triangle of exact functors
			\[\begin{tikzcd}[column sep=tiny]
				& \mathsf{Fac}T \arrow[dl,end anchor={[xshift=3ex]},"{\rm Hom}_{\mathcal{E}}(P{,}-)" swap] \arrow[dr,"{\rm Hom}_{\mathcal{E}}(T{,}-)"]& \\
				{\rm mod}_{\infty}-A\supseteq \mathsf{Fac}M \arrow[rr,"\simeq","{\rm Hom}_{A}(M{,}-)" swap] & & {_{\bot}M}
			\end{tikzcd}.\]
		\end{enumerate}
\end{cor}
\begin{proof}
	There is an equivalence 
	\[{\rm Mod}-\mathcal{P}\stackrel{\simeq}\longrightarrow {\rm Mod}-A,~F\mapsto F(P)\]
	which restricts to an equivalence
	\[{\rm mod}_{\infty}-\mathcal{P}\stackrel{\simeq}\longrightarrow {\rm mod}_{\infty}-A.\]
	Similarly we have 
	\[{\rm Mod}-\overline{\mathcal{P}}\stackrel{\simeq}\longrightarrow {\rm Mod}-A/I~\text{restricts to}~{\rm mod}_{\infty}-\overline{\mathcal{P}}\stackrel{\simeq}\longrightarrow {\rm mod}_{\infty}-A/I\]
	and
	\[\mathcal{T}-{\rm Mod}\stackrel{\simeq}\longrightarrow B-{\rm Mod}~\text{restricts to}~\mathcal{T}-{\rm mod}_{\infty}\stackrel{\simeq}\longrightarrow B-{\rm mod}_{\infty}.\]
	Thus (1) follows by Theorem \ref{thm4} (1) and (2).
	
	Consider the commutative diagram
    \[\begin{tikzcd}
    	{\rm Mod}-\mathcal{T} \arrow[r,"-\otimes_{\mathcal{T}}\Psi",shift left=0.5ex] \arrow[d,"\simeq"] & {\rm Mod}-\mathcal{P} \arrow[l,"{\rm Hom}_{\mathcal{P}}(\Psi{,}-)",shift left=0.5ex]\arrow[d,"\simeq"]\\
    	{\rm Mod}-B \arrow[r,"-\otimes_{B}M",shift left=0.5ex] & {\rm Mod}-A \arrow[l,"{\rm Hom}_{A}(M{,}-)",shift left=0.5ex]
    \end{tikzcd}.\] Then (2) follows by Theorem \ref{thm4} (3).

    We have commutative diagrams
    \[\begin{tikzcd}
    	\mathcal{E} \arrow[r,"\mathbb{P}"] \arrow[d,equal]& {\rm mod}_{\infty}-\mathcal{P} \arrow[d,"\simeq"] \\
    	\mathcal{E} \arrow[r,"(P{,}-)"] & {\rm mod}_{\infty}-A 
    \end{tikzcd}~\text{and}~\begin{tikzcd}
    \mathcal{E} \arrow[r,"\mathbb{T}'"] \arrow[d,equal]& {\rm Mod}-\mathcal{T} \arrow[d,"\simeq"] \\
    \mathcal{E} \arrow[r,"(T{,}-)"] & {\rm Mod}-B  
\end{tikzcd}.\] Therefore (3) follows by Theorem \ref{thm4} (4).
\end{proof}

\begin{rem}\label{final remark}
	(1) If $\mathcal{T}$ is a tilting subcategory, then Theorem \ref{thm4} is a part of \cite[Theorem 6.13]{Sauter}.
	
	(2) If $\mathcal{E}={\rm mod}-\Lambda$ ($\Lambda$ is an artin algebra) and $T$ is a support $\tau$-tilting module. Take $P=\Lambda $ in Corollary \ref{object version} and we have $B-{\rm mod}_{\infty}=B-{\rm mod},~{\rm mod}_{\infty}-A/I={\rm mod}-A/I,~{\rm mod}_{\infty}-A={\rm mod}-A$ and ${\rm mod}_{\infty}-B={\rm mod}-B$. Thus we obtain \cite[Proposition 3.5]{G.Jasso}.
\end{rem}

\section{Examples}

In this section, we illustrate some of our results with a simple example. 

Let $\Lambda$ be the finite dimensional algebra given by the quiver $1\stackrel{\alpha}\rightarrow 2\stackrel{\beta}\rightarrow 3$ with relation $\alpha \beta =0$. The Auslander-Reiten quiver $\Gamma({\rm mod}-\Lambda)$ and the support $\tau$-tilting quiver Q(s$\tau$-tilt$\Lambda$) of $\Lambda$ are as follows (see \cite[Example 6.4]{AIR} or \cite[Example 3.20]{G.Jasso}).
\begin{figure}[htbp]
	\begin{tikzpicture}
		\node (1) at(0,0){3};
		\node (2) at(0.8,0.8){$\begin{matrix}
				2\\3
			\end{matrix}$};
		\node (3) at(1.6,0){2};
		\node (4) at(2.4,-0.8){$\begin{matrix}
				1\\2
			\end{matrix}$};
		\node (5) at(3.2,0){1};
		\draw[->] (1)--(2);
		\draw[-,dotted] (1)--(3);
		\draw[->] (2)--(3);
		\draw[->] (3)--(4);
		\draw[-,dotted] (3)--(5);
		\draw[->] (4)--(5);
	\end{tikzpicture}
	\begin{tikzpicture}[commutative diagrams/every diagram]
		\node (a) at (90:4.6)      {$ 3 \oplus \begin{matrix}
				2\\3
			\end{matrix} \oplus \begin{matrix}
				1\\2
			\end{matrix} $};
		\node (b) at (90+72:4.6)   {$ 3 \oplus \begin{matrix}
				1\\2
			\end{matrix} \oplus 1 $};
		\node (c) at (90+72*2:4.6) {$ 3 \oplus 1 $};
		\node (d) at (90+72*3:4.6) {$ 3 $};
		\node (e) at (90+72*4:4.6) {$ 3 \oplus \begin{matrix}
				2\\3
			\end{matrix} $};
		\node (f) at (90:3)      {$ 2 \oplus \begin{matrix}
				2\\3
			\end{matrix} \oplus \begin{matrix}
				1\\2
			\end{matrix} $};
		\node (g) at (90+72:3)   {$ \begin{matrix}
				1\\2
			\end{matrix} \oplus 1 $};
		\node (h) at (90+72*2:3) {$ 1 $};
		\node (i) at (90+72*3:3) {$ 0 $};
		\node (j) at (90+72*4:3) {$ 2 \oplus \begin{matrix}
				2\\3
			\end{matrix} $};
		\node (k) at (90+36:3*0.809017)    {$ 2 \oplus \begin{matrix}
				1\\2
			\end{matrix} $};
		\node (l) at (90-72-36:3*0.809017) {$ 2 $};
		\path[commutative diagrams/.cd, every arrow]
		(a) edge (b)
		(b) edge (c)
		(c) edge (d)
		(a) edge (e)
		(e) edge (d)
		(a) edge (f)
		(f) edge (k)
		(k) edge (g)
		(g) edge (h)
		(h) edge (i)
		(f) edge (j)
		(j) edge (l)
		(l) edge (i)
		(k) edge (l)
		(b) edge (g)
		(e) edge (j)
		(d) edge (i)
		(c) edge (h);
	\end{tikzpicture}
\end{figure}

Let $T$ be the support $\tau$-tilting module $2 \oplus \begin{matrix}
	2\\3
\end{matrix} \oplus \begin{matrix}
	1\\2
\end{matrix}$. Then by Lemma \ref{enough proj T}, $\mathcal{E}:=\mathsf{Fac}T=\mathsf{add}(T\oplus 1)$ is a full exact subcategory of ${\rm mod}-\Lambda$ with enough projectives $\mathsf{add}T$. We can check directly that all basic $\tau$-rigid pairs $(U,Q)$ of $\mathcal{E}$ are as follows:
\[(0,2\oplus \begin{matrix}
	2\\3
\end{matrix} \oplus \begin{matrix}
	1\\2
\end{matrix}),~(\begin{matrix}
	2\\3
\end{matrix},2\oplus \begin{matrix}
1\\2
\end{matrix}),~(1,2\oplus \begin{matrix}
2\\3
\end{matrix}),~(2 \oplus \begin{matrix}
2\\3
\end{matrix},\begin{matrix}
1\\2
\end{matrix}),\]
\[(\begin{matrix}
2\\3
\end{matrix} \oplus \begin{matrix}
1\\2
\end{matrix},0),~(1\oplus \begin{matrix}
2\\3
\end{matrix},2),~(2\oplus \begin{matrix}
2\\3
\end{matrix} \oplus \begin{matrix}
1\\2
\end{matrix},0),~(1\oplus \begin{matrix}
2\\3
\end{matrix} \oplus \begin{matrix}
1\\2
\end{matrix},0).\]
All of them are support $\tau$-tilting pairs except $(\begin{matrix}
	2\\3
\end{matrix} \oplus \begin{matrix}
	1\\2
\end{matrix},0)$, which can be completed to a support $\tau$-tilting pair $(1\oplus \begin{matrix}
2\\3
\end{matrix} \oplus \begin{matrix}
1\\2
\end{matrix},0)$ (see Proposition \ref{inclusion}). We can also check Corollary \ref{Cor:number of ind} easily. The Hasse quiver Q(s$\tau$-tilt$\mathcal{E}$) of the poset of all support $\tau$-tilting subcategories (or objects) of $\mathcal{E}$ (see Corollary \ref{poset}) is as follows.
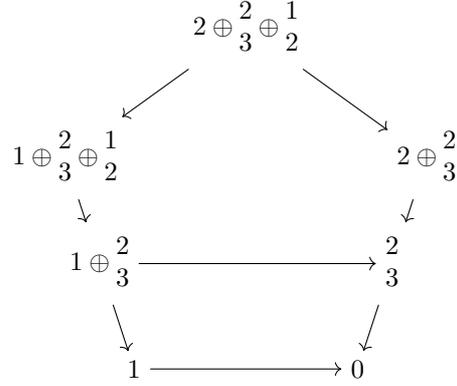
\begin{figure}[htbp]
	\begin{tikzpicture}[commutative diagrams/every diagram]
		\node (f) at (90:2.5)      {$ 2 \oplus \begin{matrix}
				2\\3
			\end{matrix} \oplus \begin{matrix}
				1\\2
			\end{matrix} $};
		\node (g) at (90+72:2.5)   {$1\oplus \begin{matrix}
				2\\3
			\end{matrix}\oplus \begin{matrix}
				1\\2
			\end{matrix}$};
		\node (h) at (90+72*2:2.5) {$ 1 $};
		\node (i) at (90+72*3:2.5) {$ 0 $};
		\node (j) at (90+72*4:2.5) {$ 2 \oplus \begin{matrix}
				2\\3
			\end{matrix} $};
		\node (k) at (90+72+36:2.5*0.809017)    {$ 1 \oplus \begin{matrix}
				2\\3
			\end{matrix} $};
		\node (l) at (90-72-36:2.5*0.809017) {$\begin{matrix}
				2\\3
			\end{matrix}$};
		\path[commutative diagrams/.cd, every arrow]
		(f) edge (g)
		(k) edge (h)
		(g) edge (k)
		(h) edge (i)
		(f) edge (j)
		(j) edge (l)
		(l) edge (i)
		(k) edge (l);
	\end{tikzpicture}
\end{figure}

Note that all support $\tau$-tilting modules of ${\rm mod}-\Lambda$ in $\mathcal{E}$ are as follows:
\[2\oplus \begin{matrix}
	2\\3
\end{matrix} \oplus \begin{matrix}
	1\\2
\end{matrix},~2\oplus \begin{matrix}
1\\2
\end{matrix},~\begin{matrix}
1\\2
\end{matrix}\oplus 1,~2\oplus \begin{matrix}
2\\3
\end{matrix},~1,~2,~0.\]
Thus in general, for an artin algebra $\Lambda$ and a support $\tau$-tilting module $T$, support $\tau$-tilting modules of ${\rm mod}-\Lambda$ in $\mathsf{Fac}T$ are not support $\tau$-tilting objects of $\mathcal{E}=\mathsf{Fac}T$ and vice versa. For instance, support $\tau$-tilting module $\begin{matrix}
	1\\2
\end{matrix}\oplus 1$ is a $\tau$-rigid object of $\mathcal{E}$ but does not satisfy (A). Because left $\mathsf{add}(\begin{matrix}
1\\2
\end{matrix}\oplus 1)$ approximation of $\begin{matrix}
2\\3
\end{matrix}$ is $\begin{matrix}
2\\3
\end{matrix} \rightarrow \begin{matrix}
1\\2
\end{matrix}$, which is not admissible since its kernel is $3\notin \mathcal{E}$. For support $\tau$-tilting object $1\oplus \begin{matrix}
2\\3
\end{matrix}$ of $\mathcal{E}$, it is not $\tau$-rigid in ${\rm mod}-\Lambda$. However, we have a simple observation as follows. 

\begin{cor}
	Let $\Lambda$ be an artin algebra and $T,T'$ be support $\tau$-tilting modules such that $T'\leq T$. If the kernel of left $\mathsf{add}T'$-approximation of $T$ is in $\mathsf{Fac}T$, then $|T'|+|\tilde{T}|=|T|$ where $\tilde{T}$ is a maximal summand of $T$ such that ${\rm Hom}_{\Lambda}(\tilde{T},T')=0$.
\end{cor}
\begin{proof}
	Let $\mathcal{E}:=\mathsf{Fac}T$ be the full exact subcategory with enough projectives $\mathsf{add}T$. Since $T'\in \mathcal{E}$ and ${\rm Ext}_{\Lambda}^{1}(T',\mathsf{Fac}T')=0$, then ${\rm Ext}_{\mathcal{E}}^{1}(T',\mathsf{Fac}_{\mathcal{E}}T')=0$. Consider the exact sequence $$T\stackrel{f}\rightarrow T^{1}\rightarrow T^{2}\rightarrow 0$$ with $f$ a left $\mathsf{add}T'$-approximation. By Lemma \ref{enough proj T} and $T'\leq T$, we have $f$ is also a left $\mathsf{Fac}T'$-approximation. Apply ${\rm Hom}_{\Lambda}(-,\mathsf{Fac}T')$ to the sequence, we deduce $T^{2}\in \mathsf{add}T'$ as in the proof of Proposition \ref{inclusion}. Thus the sequence is exact in $\mathcal{E}$. Therefore $T'$ is a support $\tau$-tilting object in $\mathcal{E}$ and then the result follows by Theorem \ref{thm:number component}.
\end{proof}

\bibliographystyle{plain}

\end{document}